\documentclass[11pt]{amsart}
\usepackage{amsmath, amsthm, amssymb}
\usepackage{amsmath,amscd}
\usepackage{mathabx}
\usepackage{hyperref}
\usepackage{mathrsfs}
\usepackage{xcolor, changepage} 
\usepackage{graphicx}
\usepackage{titletoc}
\usepackage{enumerate}
\usepackage{accents}

\usepackage{tikz}
\usetikzlibrary{matrix,arrows}
\usetikzlibrary{shapes}
\usetikzlibrary{calc}
\usetikzlibrary{arrows}
\usetikzlibrary{decorations.pathreplacing,decorations.markings}
\usepackage[all]{xy}
\usepackage{tikz-cd}

\usepackage{caption}
\usepackage{subcaption}
\usepackage{geometry}
\theoremstyle{plain}
\newtheorem{theorem}{Theorem}
\newtheorem{proposition}[theorem]{Proposition} 
\newtheorem{lemma}[theorem]{Lemma}
\newtheorem{remark}[theorem]{Remark}
\newtheorem{corollary}[theorem]{Corollary}
\newtheorem{definition}[theorem]{Definition}

\numberwithin{theorem}{section}

\DeclareMathOperator{\Lip}{Lip}
\DeclareMathOperator{\dist}{dist}

\DeclareMathOperator{\spt}{spt}
\DeclareMathOperator{\Hess}{Hess}
\DeclareMathOperator{\ggraph}{graph}
\DeclareMathOperator{\reg}{reg}

\numberwithin{equation}{section}

\title[foliation in AF manifolds]{foliation of area-minimizing hypersurfaces in asymptotically flat manifolds of higher dimension}
\author{Shihang He}
\address{Key Laboratory of Pure and Applied Mathematics,
	School of Mathematical Sciences, Peking University, Beijing, 100871, P. R. China
}
\email{hsh0119@pku.edu.cn}
\author{Yuguang Shi}
\address{Key Laboratory of Pure and Applied Mathematics,
	School of Mathematical Sciences, Peking University, Beijing, 100871, P. R. China
}
\email{ygshi@math.pku.edu.cn}

\author{Haobin Yu}
\address{
	School of Mathematics, Hangzhou Normal University, Hangzhou, 311121, P. R. China
}
\email{yhbmath@hznu.edu.cn}
\thanks{S. He, Y. Shi  are funded by the National Key R\&D Program of China Grant 2020YFA0712800 and NSFC12431003. H. Yu is funded by Zhejiang Provincial NSFC No. LY24A010008}

\subjclass[2010]{Primary 53C21, secondary 53C24 }

\begin{document}
\begin{abstract}
We prove the existence of foliations by area-minimizing hypersurfaces in asymptotically flat (AF) manifolds with arbitrary dimension and 	arbitrary ends.
Also we provide behaviors of those hypersurfaces near the infinity of AF ends and demonstrate that the singular set of those area-minimizing hypersurfaces is outside AF ends (cf Theorem \ref{thm: foliation}). Building on the positive mass theorem for AF manifolds with arbitrary ends, we establish a global behavior for free-boundary area-minimizing hypersurfaces inside coordinate cylinders in AF manifolds  of dimension less than or equal to $8$ (cf. Theorem \ref{thm: 8dim Schoen conj}).
\end{abstract}
	\maketitle
	\tableofcontents	
	
	\section{Introduction}
	
	\subsection{Foliations of Area-minimizing hypersurfaces in General AF manifolds}

    $\quad$
	
	In  \cite{HSY24}, foliations of  area-minimizing hypersurfaces in asymptotically flat  (AF) manifolds of dimensions no greater than $7$ were obtained. Such foliations play an important role in the study of effective versions of positive mass theorem (PMT). In this paper we generalize this to higher dimensions. Namely, we  prove the following:

	\begin{theorem}\label{thm: foliation}
		Let $(M^{n+1},g)$ be an AF manifold of dimension $n\geq 3$ with an AF end $E$ of  order $\tau>\frac{n}{2}$ and some arbitrary ends. Suppose $(M^{n+1},g)$ is geometric bounded, i.e. its sectional curvature is bounded, injective radius has a positive uniform lower bound. Then 
		
		(1) For each $t\in\mathbf{R}$, there exists an area-minimizing hypersurface $\Sigma_t$  which is asymptotic to the coordinate hyperplane \{z=t\}. Moreover, there exists a compact set $\mathbf{K}$ depending only on the geometry of $(E,g|_E)$ (independent of $t$), such that $(\Sigma_t\setminus\mathbf{K})\cap E$ can be represented by some graph, $i.e.$ 
        $$(\Sigma_t\setminus\mathbf{K})\cap E = \{(y,u_t(y)):y\in\mathbf{R}^{n} \setminus \mathbf{D}\}$$ 
        for each $t\in\mathbf{R}$, where  $\mathbf{D}$ is a compact set of $\mathbf{R}^{n}$ and $u_t\in C^\infty(\mathbf{R}^n\setminus\mathbf{D})$. In particular, $\Sigma_t$ is smooth outside $\mathbf{K}$. 
		
		(2) There is a constant $C$ depending only on $k$ and $(M,g)$, such that for any  $k\ge 1$, $\epsilon>0$, the graph function satisfies 
		$|u_t-t|+|y|^k|D_ku_t(y)|\leq C|y|^{1-\tau-k+\epsilon}$ as $|y|\to +\infty$.
		
		(3) There exists a constant $T>0$ such that any area-minimizing hypersurface $\Sigma_t$ with $|t|>T$ is smooth and any area-minimizing hypersurface satisfying the condition above is unique. Furthermore, the region above $\Sigma_T$(or below $\Sigma_{-T}$) can be $C^1$ foliated by $\{\Sigma_t\}$.
		
		Moreover, if $M$ is an  AF manifold with some arbitrary ends and without geometric bounded condition, then the conclusions $(1)-(3)$ hold for $t$ with $|t|>T_0$ for some sufficiently large $T_0$.
	\end{theorem}

	\begin{remark}\
	\begin{itemize}
		\item The existence of area-minimizing hypersurface $\Sigma_t$ for each $t$ was known by Schoen-Yau in \cite{SY79PNAS};
		\item Let $(M^{n+1},g)$ be as in Theorem \ref{thm: foliation} and $\Sigma'_t$ be any area-minimizing hypersurface asymptotic  the coordinate hyperplane \{z=t\} for some $|t|\leq T$. Then $\Sigma'_t$ may contain singular set $\mathcal{S}$. However, we are able to show that $\mathcal{S}$ is contained in a  fixed compact set $\mathbf{K}$ of $M^{n+1}$ (see Corollary \ref{cor: bound ms} below) and $\Sigma'_t$ has the same decay property as (2).  More generally, if $M$ is an  AF manifold with some arbitrary ends  but without geometric bounded condition, let $\Sigma'_t$ be an area-minimizing hypersurface in $M$ that is asymptotic to the hyperplane $\{z = t\}$, we can show that $\mathcal{S}\subset \mathbf{K}\cup (M\backslash E)$, where $\mathbf{K}\subset E$ is a compact set. In particular,  if $n=7$, by  geometric measure theory, we know that for any compact set $\mathbf{L}\subset M^{n+1}$, $\mathcal{S}\cap \mathbf{L}$ consists of finitely many points. 
        	\end{itemize}
\end    {remark}
		  It would be interesting to compare the progress in constant mean curvature(CMC) foliations with the area-minimizing foliation constructed here. In both settings the smooth foliation exists in all dimensions in the asymptotic region. For an AF manifold of dimension $n\geq3$ with
positive mass, Eichmair-Koerber \cite{EK24} established the existence of an
asymptotic foliation of $(M, g)$ by stable CMC spheres based on Lyapunov-Schmidt reduction. See also Eichmair-Metzger \cite{EM13} for the case of isoperimetric foliations in asymptotically Schwarzschild manifolds with positive mass.

In \cite{HSY24}, we showed that the existence and behavior of a certain class of area-minimizing hypersurfaces within an exhaustive sequence of coordinate cylinders \footnote{A coordinate cylinder of radius $R$ in an AF end is given by $C_R=\{ (x_1,x_2,\dots,x_n,z)\in \mathbf{R}^{n+1}, x_1^2+x_2^2+\dots+x_n^2< R^2\}.$} in an AF manifold of dimension less than or equal to $7$ and with non-negative scalar curvature is heavily dependent on the  mass of the AF manifold (cf. Theorem 1.3 in \cite{HSY24}).	Such a phenomenon can be regarded as an effective version of positive mass theorem (PMT) of the AF manifold, by which one get  PMT directly.   In this paper, we continue this investigation in higher dimensional  AF manifolds $(M^{n+1},g)$, possibly with arbitrary ends.
	
	\begin{theorem}\label{thm: 8dim Schoen conj}
		Let $(M^{n+1},g)$ be an AF manifold with arbitrary ends and  a distinguished AF end  $E$ of asymptotic order $\tau>max \{\frac{n}{2}, n-2\}$ $(n+1\le 8)$. Suppose the scalar  curvature $R_g\geq 0 $ on $M^{n+1}$,  and  $m(M, g, E) \neq 0$.

        Furthermore, assume that $(M^{n+1},g)$ falls into one of the following two cases:
\begin{enumerate}
  \item $M$ has only AF ends.
  \item $M$ may have arbitrary ends other than $E$. Moreover, there exist two neighborhoods $U_1$ and $U_2$ of $E$ with
  \[
  E \subset U_1 \subset U_2,
  \]
  such that $\overline{U_i\setminus E}$ is compact and $\partial U_i$ is smooth for $i=1,2$, and
  \[
  R_g>0 \quad \text{in } U_2\setminus U_1.
  \]
\end{enumerate}
        
        Then one of the following happens:
		\begin{itemize}
			\item There exists $R_0>0$ such that for all $R>R_0$ there exists no element $\Sigma_R$, which minimizes the volume in the  $\mathcal{F}_R$ (see \eqref{eq: 79}  for detailed definitions);
			\item For any sequence $\{R_i\}$ tending to infinity such that there exists a sequence of hypersurfaces $\Sigma_{R_i}$  minimizing the volume in $\mathcal{F}_{R_i}$, we have that $\Sigma_{R_i}$ drifts  to the infinity,i.e. for any compact set $\Omega \subset M$, 	$\Sigma_{R_i}\cap \Omega= \emptyset$ for sufficiently large $i$.
		\end{itemize}
	\end{theorem}
	\begin{remark}
			 \item If $(M^{n+1},g)$  admits  only AF ends (the case (1) in Theorem \ref{thm: 8dim Schoen conj}), Theorem \ref{thm: 8dim Schoen conj} can be deduced  from Theorem 1.6  of \cite{CCE16}  for AF ends of asymptotic order $\tau>\frac{1}{2}$ when $n=2$ and Theorem 2 of \cite{Carlotto16} for  asymptotically Schwarzschild manifolds  when $3 \leq n \leq 6$. 
	\end{remark}

	Theorem \ref{thm: 8dim Schoen conj} reflects certain global effects of the mass for AF manifolds with nonnegative scalar curvature and   arbitrary ends. We would like to emphasize that, even when all ends are AF, an effective form of the positive mass theorem for AF manifolds (as stated in Theorem \ref{thm: 8dim Schoen conj}) appears to be new in dimensions greater than 7. The sequence of minimal hypersurfaces provides a quantitative characterization of the positivity of the mass. In the case that $(M^{n+1},g)$ possesses arbitrary ends, Theorem \ref{thm: 8dim Schoen conj} (2) can be regarded as a \textit{ shielded phenomenon for free boundary minimizing hypersurfaces}. This could be compared with the shielded version of the positive mass theorem proved in \cite{LUY21}.

    \begin{figure}
            \centering
            \includegraphics[width=15cm]{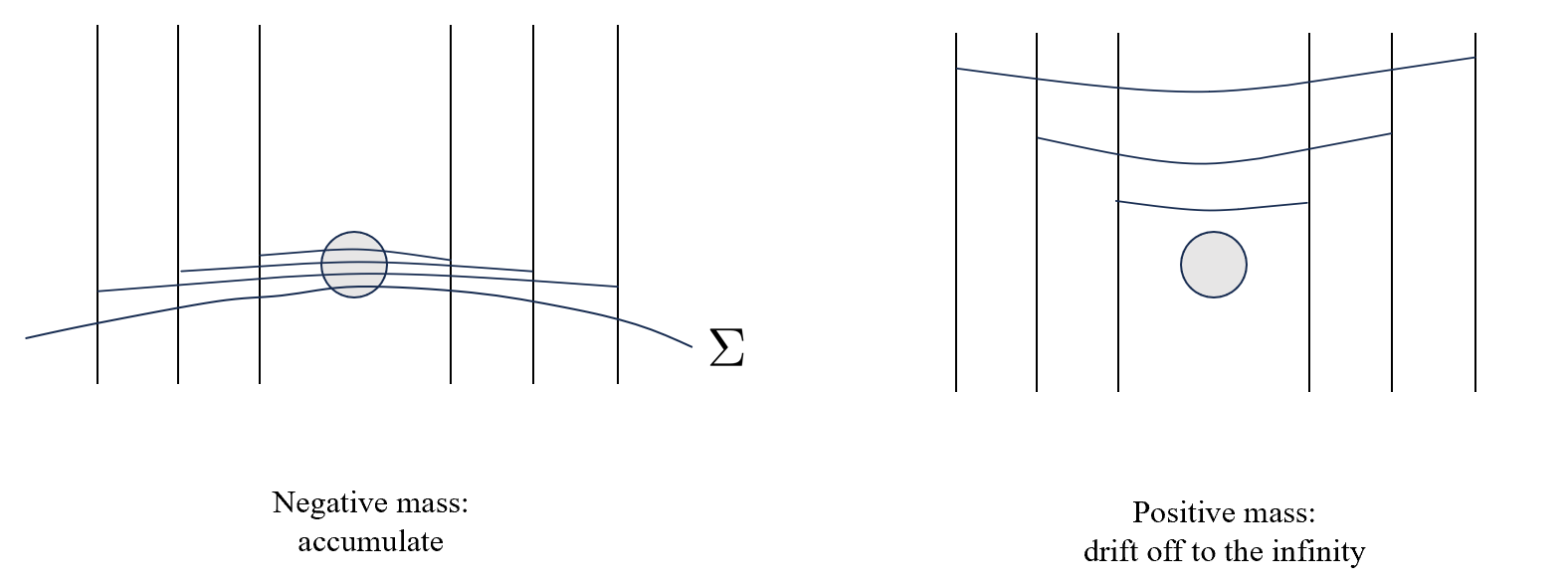}
            \caption{An illustration of Theorem \ref{thm: 8dim Schoen conj}: Vertical pairs of lines represent coordinate cylinders, while horizontal curves represent free boundary minimizing hypersurfaces.}
            \label{c2}
        \end{figure}
	
\subsection{Outline of Proofs}

We focus on area-minimizing boundaries in AF manifolds. Following  \cite{HSY24}, we select a sequence of exhausting coordinated cylinder $C_{r_i}\subset C_{r_{i+1}}$  in the AF end of  $(M^{n+1},g)$, and solve the Plateau's problem in $C_{r_i}$ with  an $(n-1)$-dimensional coordinated sphere within $C_{r_i}$ as boundary datum. Here $r_i$ denotes the coordinated  radius of $C_{r_i}$ and $r_i\rightarrow \infty$ as $i\rightarrow \infty$. Note that the solution to our Plateau's problem is  an area-minimizing hypersurface  that forms   boundary of a Caccioppoli set. Then by comparing with the coordinate ball enclosed by the $(n-1)$-dimensional coordinated sphere in $C_{r_i}$, we observe that  for $i$ large enough, the volume density of the solution of our Plateau's problem can be arbitrarily close to $1$ outside a fixed compact set of $M^{n+1}$ (see Proposition \ref{prop: density estimate1} below). Consequently, invoking of Allard's regularity theorem in geometric measure theory, we deduce that the singular set of $\Sigma_t$ in Theorem \ref{thm: foliation} is contained within a fixed compact set of $M^{n+1}$ (see Corollary \ref{cor: bound ms} below). The foliation structure of the AF end follows from the similar argument in \cite{HSY24}.

    For the  proof of Theorem \ref{thm: 8dim Schoen conj}, we adopt the contradiction argument.
    Suppose there exist a sequence of area minimizing boundary $\Sigma_i$ in free boundary sense
    such that $\Sigma_i\cap K\neq\emptyset$ for some fixed compact $K$. Then $\Sigma_i$ converge to a strongly stable area minimizing hypersurface. We can construct  a strongly stable area minimizing boundary $\Sigma_p$ passing through $p$ for any $p\in E$ with $|z(p)|\gg1$ using the method of \cite{Liu13,CCE16}. In particular, $\Sigma_p\cap E$ is asymptotic to some coordinate hyperplane.  By combining Theorem 1.1 in \cite{HSY26} with Theorem \ref{thm: foliation}, we show the mass of the AF end $E$ of $(M,g)$ equals $0$, which gives a contradiction.

	\subsection{Organization of The Paper}
	In Section 2, we introduce some key definitions and preliminary estimates for area-minimizing hypersurfaces, and then prove Theorem \ref{thm: foliation} at the end of this section.  Section 3 is devoted to the proof of Theorem \ref{thm: 8dim Schoen conj}.

\section{Proof of Theorem \ref{thm: foliation}}

In this section, we consider some basic estimates for area minimizing hypersurfaces in the coordinate cylinder $C_r$ in an asymptotically flat (AF) manifold $(M^{n+1}, g)$ of dimension $n+1$.  Using these estimates, we provide a proof of Theorem \ref{thm: foliation} at the end of this section. 

Let us begin with the following definition of AF manifolds.

\begin{definition}\label{def: AF manifold}
		An end $E$ of an $(n+1)$-dimensional Riemannian manifold $(M^{n+1},g)$ is said to be asymptotically flat (AF) of order $\tau$ for some $\tau>\frac{n-1}{2}$, if $E$ is diffeomorphic to $\mathbf{R}^n\backslash B^{n+1}_{1}(O)$, and the metric $g$ in $E$ satisfies
		\begin{align*}
			|g_{ij}-\delta_{ij}|+|x||\partial g_{ij}|+|x|^2|\partial^2 g_{ij}| = O(|x|^{-\tau}),
		\end{align*}
		and $R_g\in L^1(E)$. The  mass of $(M,g,E)$ is defined by
            \begin{align*}
                m_{ADM}(M,g,E) = \frac{1}{2n\omega_{n}}\lim_{\rho\to\infty}\int_{S^{n}(\rho)}(\partial_ig_{ij}-\partial_j g_{ii})\nu^jd\sigma_x
            \end{align*}
  \end{definition}

    In the case above we always call $(M,g)$ an asymptotically flat manifold: $(M,g)$ has an distinguished AF end $E$ and some arbitrary ends. Throughout the paper, we always assume
    the AF end of a  Riemannian manifold $(M^{n+1},g)$  is of order $\tau$ for some $\frac{n-1}{2}<\tau\leq n-1$.

Let us collect the notations that will be used throughout our discussion:
\begin{itemize}
    \item $B_r^{n+1}(p)$ (or simply $B_r(p)$) denotes the coordinate ball $\{x\in \mathbf{R}^{n+1}: |x-p|<r\}$;
    \item $\mathcal{B}^{n+1}_r(p)$ denotes a geodesic ball of center $p$ and radius $r$ in the Riemannian manifold $M^{n+1}$;
    \item $(y,z),y\in\mathbf{R}^n,z\in\mathbf{R}$  denotes points in the AF end $E\cong \mathbf{R}^{n+1}\backslash B_1^{n+1}(O)$;
    \item $S_t$ denotes the coordinate $z$-hyperplane $\{(y,z): z = t\}$ in $E$;
    \item  Within the AF end $E$,  the hypersurface $(\partial B^{n+1}_R(O) \cap S_0) \times \mathbf{R}$ partitions $E$ into two distinct regions: an inner part and an outer part. Let $C_r$ be the  inner part, which we call the \textit{cylinder of radius} $r$.
\end{itemize}

To better illustrate the ideas, throughout this section our AF manifold $M^{n+1}$ is assumed to have a single end, and at the end of this section we will explain how to extend the results to AF manifolds with arbitrary ends by using exactly the same idea; See also Remark \ref{rmk: no drift to infinity for arbitrary ends}.

The rest of this section is organized as follows.
In Section 2.1, we prove the existence of a solution $\Sigma_{a,t}$ to the Plateau problem with boundary $S_t\cap \partial C_a$.
In Section 2.2, we establish a density estimate for $\Sigma_{a,t}$ outside a compact set $B_{a_0}(O)$; see Proposition \ref{prop: density estimate1}.
In Section 2.3, we first prove the smoothness of $\Sigma_{a,t}$ outside a compact set; see Corollary \ref{cor: smooth outside compact set}.
We then establish a curvature estimate for $\Sigma_{a,t}$; see Lemma \ref{lem: uniform estimate for A}.
Finally, in Section 2.4 we prove Theorem \ref{thm: foliation}.

\subsection{Plateau problem in the cylinder $C_a$}

$\quad$

We aim to solve the Plateau problem in the cylinder $C_a$ with prescribed boundary 
\[
S_{a,t}:=\partial C_a\cap S_t .
\]
We reformulate the problem in the language of geometric measure theory.

\begin{proposition}\label{prop: solving plateau}
    Let $\Omega_{a,t}$ be a Caccioppoli set such that $\Omega_{a,t}\cap (E\backslash C_a) = \{z\le t\}\cap (E\backslash C_a)$. Then for $a$ sufficiently large, there exists a Caccioppoli set $E_{a,t}$, such that
    \begin{align}\label{eq: perimeter minimizing}
        P(E_{a,t}, C_a)\le P(F,C_a)
    \end{align}
    for all Caccioppoli set $F$ satisfying $F = \Omega_{a,t}$ in $E\backslash C_a$. Here $P(F,C_a)$ denotes the perimeter of $F$ in $C_a$.

    Furthermore, let $\Sigma_{a,t} = \partial^*E_{a,t}\cap C_a$, where $\partial^*E_{a,t}$ is the reduced boundary of $E_{a,t}$. Then $\Sigma_{a,t}$ has a multiplicity one integral current structure as well as a multiplicity one rectifiable varifold structure, satisfying $\Sigma_{a,t}\cap \partial C_a = S_{a,t}$ and $\partial[|\Sigma_{a,t}|] = [|S_{a,t}|]$. 
\end{proposition}
\begin{proof}
    The Existence of $E_{a,t}$ is a direct consequence of \cite[Theorem 1.20]{Giu1984}. The current and varifold structure of $\Sigma_{a,t}$ is a consequence of De Giorgi structure theorem (See for example \cite[Theorem 14.3]{Simon83}). To see $\Sigma_{a,t}\cap \partial C_a = S_{a,t}$, we note from the asymptotical flatness of $E$ that $\partial C_s (s\ge a)$ forms a mean convex foliation as long as $a$ is sufficiently large. Thus, this follows from the maximum principle. Finally, for $\omega\in \Omega^{n-1}_c(C_a)$, it holds $\partial[|\Sigma_{a,t}|](\omega) = [|\Sigma_{a,t}|](d\omega) = \partial [|E_{a,t}|](d\omega) = 0$, which implies $\spt(\partial[|\Sigma_{a,t}|])\subset E\backslash C_a$. Combined with $\Sigma_{a,t}\cap \partial C_a = S_t\cap C_a = S_{a,t}$ and the constancy theorem \cite[Theorem 26.27]{Simon83}, it follows that $\partial[|\Sigma_{a,t}|] = [|S_{a,t}|]$.
\end{proof}

\subsection{Density estimate for area-minimizing  hypersurfaces  in asymptotically flat manifolds}

$\quad$

In this subsection, we derive the density estimate for area-minimizing  hypersurfaces  in asymptotically flat manifolds. As the first step, We shall establish monotonicity inequalities for minimal hypersurfaces in both local and global cases.

$\quad$

\textbf{1.Local inequalities}

Choose some fixed $r_0>1$ (Recall "$1$" is the number appeared in Definition \ref{def: AF manifold}) such that the following holds
\begin{itemize}
    \item Outside $B^{n+1}_{r_0}(O)$, the injective radius of $(M^{n+1},g)$ is at least $1$, and the sectional curvature satisfies $|sec|\leq k_0$. 
\end{itemize}

\begin{lemma}\label{lm: monotonicity ineq2} 
 Let $\Sigma^{n}$
be an area-minimizing hypersurface in $(M^{n+1},g)$. 
Then for all $\xi\in  M$ satisfying 
$\mathcal{B}^{n+1}_1(\xi)\subset M^{n+1}\setminus B^{n+1}_{r_0}(O) $, there holds
\begin{equation}\label{eq: monotonicity ineq2}
    e^{k_0\rho_2}\rho_2^{-n}\mathcal{H}^{n}(\mathcal{B}^{n+1}_{\rho_2}(\xi)\cap \Sigma)
    \geq e^{k_0\rho_1}\rho_1^{-n}\mathcal{H}^{n}(\mathcal{B}^{n+1}_{\rho_1}(\xi)\cap \Sigma)
\end{equation}
for any $0<\rho_1<\rho_2<1$. Here $\mathcal{B}^{n+1}_{\rho}(\xi)$ denotes the geodesic ball in $(M^{n+1},g)$.
\end{lemma}
\begin{proof}
    Let $r$ be the distance of $x$ to $\xi$  in $(M^{n+1},g)$. Then
    \begin{equation}
        div_{\Sigma}(\nabla^M r^2)=\sum_{i=1}^{n}2r\Hess^Mr(e_i,e_i)
        +2\langle \nabla^Mr, e_i\rangle^2.
    \end{equation}
    By Hessian comparison theorem (see \cite[p.234]{CM11}) we know in 
    $\mathcal{B}^{n+1}_1(\xi)$, for any vector  $V$
with  $|V|=1$, there holds
\[
| \Hess^M  r (V,V)-\frac{1}{r}(1-\langle V, \nabla^M r\rangle^2)|\leq\sqrt{k_0}.
\]
It follows
\[
|div_{\Sigma}(\nabla^M r^2)-2n|\leq2n\sqrt{k_0}r.
\]
The rest of the proof follows from the standard argument.
\end{proof}
\textbf{2. Global inequalities}

In order to get the regularity of area-minimizing hypersurfaces, we need a global monotonicity formula for volume density as well. 
 Let $\xi=(\xi^1, \cdots, \xi^{n+1})$ be a  point in the AF end $E$ with $|\xi|> R_0\geq r_0$.
 Let $\{x^i\}$($1\leq i\leq n+1$) be the coordinate function on the AF end $E$
and  $f^i$ be a smooth  extension of $x^i-\xi^i$ satisfying
\begin{equation}
f^i=\left\{
\begin{aligned}
&x^i-\xi^i, \quad \text{for $x\in E\setminus B^{n+1}_{R_0}(O)$},\\
&O(|\xi|),\quad \text{for $x\in M^{n+1}\backslash (E\setminus B^{n+1}_{R_0}(O))$}.
\end{aligned}
\right.\nonumber
\end{equation}
We define
$$
r^2=(f^1)^2+\cdots +(f^{n+1})^2.
$$
A suitable extension of $f$ in $M^{n+1}\backslash (E\setminus B^{n+1}_{R_0}(O))$ ensures  $r>0$ on $M^{n+1}\setminus \{\xi\}$.
We can fix some $R_0$ such that
$|\nabla_Mr|>0$ in $E\setminus (B^{n+1}_{R_0}(O)\cup\{\xi\})$.
Given $t$, we use
 $\Sigma_{a,t}$ to denote the solution of Plateau's problem in the coordinate cylinder $C_a$ with $\partial \Sigma_{a,t}= S_t \cap C_a$ given by Proposition \ref{prop: solving plateau}.
 For simplicity, we set 
 \[
 \theta_{\xi}(\rho,a)=\frac{\mathcal{H}^{n}(\Sigma_{a,t}\cap B^{n+1}_{\rho}(\xi))}{\omega_{n}\rho^{n}},
 \]
where $B^{n+1}_{\rho}(\xi):=\{x\in M^{n+1}: |x-\xi|<\rho\}$  denotes the coordinate ball with centre $\xi$ and radius $\rho$. 

We start with the following monotonicity formula:
\begin{lemma}\label{lmm: monotonicity ineq2}
For any $0<\rho_1<\rho_2$ with 
 $(B^{n+1}_{\rho_2}(\xi)\setminus B^{n+1}_{\rho_1}(\xi))\subset M\backslash B^{n+1}_{R_0}(O)$ and $B^{n+1}_{\rho_2}(\xi)\cap\partial\Sigma_{a,t}=\emptyset$,
it holds
\begin{equation}\label{eq: monotonicity formula}
	\begin{split}
	 \theta_{\xi}(\rho_2,a)- \theta_{\xi}(\rho_1,a)
    \geq
	\int^{\rho_2}_{\rho_1}s^{-n-1}\int_{\Sigma_{a,t}\cap B^{n+1}_{s}(\xi)}(u-n|v|)ds-\rho_2^{-n}\int_{\Sigma_{a,t}\cap  B^{n+1}_{\rho_2}(\xi)}|v|,
    \end{split}
\end{equation}
where $u,v$ satisfy
\begin{equation}
 u(x)=\left\{
\begin{aligned}
&O(|x|^{-1-\tau})|\xi|+O(|x|^{-\tau}), \quad \text{for $x\in M\setminus B^{n+1}_{R_0}(O)$, }\\
&O(1)|\xi|,\quad\ \quad \quad\quad\quad\quad\quad\ \ \  \ \text{for $x\in B^{n+1}_{R_0}(O)$. }
\end{aligned}
\right.\nonumber
\end{equation}
and
\begin{equation}
 v(x)=\left\{
\begin{aligned}
&O(|x|^{-\tau}), \quad\text{for $x\in M\setminus B^{n+1}_{R_0}(O)$, }\\
&O(1),\quad\ \quad \   \text{for $x\in B^{n+1}_{R_0}(O)$. }
\end{aligned}
\right.\nonumber
\end{equation}
\end{lemma}

\begin{proof}
The arguments are almost the same as those of 	Theorem 17.6 in \cite{Simon83}.
We include a proof here for completeness. 
Let $\{\phi_k\}$ be a sequence of smooth and decreasing functions given by

\begin{equation}
\phi_k(t)=\left\{
\begin{aligned}
&1, \quad \text{for $t\leq \frac{k-1}{k}$, }\\
&0,\quad \text{for $t\geq 1$. }
\end{aligned}
\right.\nonumber
\end{equation}
For any $r\leq \rho$ we define
$$
\gamma(r):=\phi_k(\frac{r}{\rho}),\ \ 
I(\rho)=\int \phi_k(\frac{r}{\rho}) d\mu,
$$
and 
$$
J(\rho)=\int \phi_k(\frac{r}{\rho})|(\nabla_M r)^{\perp}|^2 d\mu,
$$
where $\mu$ is a Radon measure associated with $\Sigma_{a,t}$.
Namely, for any $\mathcal{H}^{n}$-measurable $A$, 
$$
\mu(A):=\mathcal{H}^{n}(\Sigma_{a,t}\cap A).
$$
By a direct computation, we have
$$
r\gamma'(r)=-\rho\frac{\partial}{\partial \rho}[\phi_k(\frac{r}{\rho})],\ \ \ \ 
I'(\rho)=-\rho^{-1}\int r\gamma'(r)  d\mu,
$$
and
$$
J'(\rho)=-\rho^{-1}\int r\gamma'(r)|(\nabla_Mr)^{\perp}|^2 d\mu.
$$
Set 
$$
Y:=\frac{1}{2}\nabla_M f
$$
and let $\{e_i\}(1\leq i\leq n+1)$ be the orthonormal frame of  $M^{n+1}$ where $\{e_i\}$,$1\leq i\leq n$ be the tangential vectors of $\Sigma_{a,t}$,  then
\begin{align}\label{eq: divY}
div_{\Sigma_{a,t}} Y&=\frac{1}{2}\sum^{n}_{i=1}\nabla^2_M f(e_i,e_i)=n+u,
\end{align}
where
\begin{equation}
 u(x)=\left\{
\begin{aligned}
&O(|x|^{-1-\tau})|\xi|+O(|x|^{-\tau}), \quad \text{for $x\in M\setminus B^{n+1}_{R_0}(O)$ },\\
&O(1)|\xi|,\quad\ \qquad \qquad \qquad  \ \ \  \ \text{for $x\in B^{n+1}_{R_0}(O)$ }.
\end{aligned}
\right.\nonumber
\end{equation}
Let $X:=\gamma(r)Y$, then
\begin{equation}
	\begin{split}
div_{\Sigma_{a,t}} X&=n\gamma(r)+u\gamma(r)+\gamma'(r) g(\nabla_{\Sigma_{a,t}} r, Y)	\\
&=n\gamma(r)+\gamma'(r)g(\nabla_{\Sigma}r,r\nabla_Mr)+u\gamma(r)\\
&=n\gamma(r)+r\gamma'(r)(|\nabla_Mr|^2-(e_{n+1}(r))^2)+u\gamma(r)\\
&=n\gamma(r)+r\gamma'(r)-r\gamma'(r)(e_{n+1}(r))^2+r\gamma'(r)v+ u\gamma(r),
\end{split}\nonumber
\end{equation}
where $v$ satisfies
\begin{equation}
 v(x)=\left\{
\begin{aligned}
&O(|x|^{-\tau}), \quad \text{for $x\in M\setminus B^{n+1}_{R_0}(O)$, }\\
&O(1),\quad\ \quad  \ \text{for $x\in B^{n+1}_{R_0}(O)$. }
\end{aligned}
\right.\nonumber
\end{equation}
Note that $\Sigma_{a,t}$ is area-minimizing, we have
$$
\int div_{\Sigma_{a,t}} X d\mu=0.
$$
Hence,
$$
\int(n\gamma(r)+r\gamma'(r))d\mu=\int \gamma'(r)r(e_{n+1}(r))^2 d\mu -\int [r\gamma'(r)v+u\gamma(r)]d\mu.
$$
Or equivalently, 
\begin{equation}\label{eq: differential on I}
\frac{d}{d\rho}(\rho^{-n}I(\rho))=\rho^{-n}J'(\rho)+\rho^{-n}L'(\rho)+\rho^{-n-1}\int u\gamma(r)d\mu.	
\end{equation}
Here 
$$
L(\rho)=-\int \phi_k(\frac{r}{\rho})vd\mu.	
$$
Taking integral of \eqref{eq: differential on I} on $[\rho_1, \rho_2]$ and setting $k\rightarrow \infty$
yields 
\begin{equation}\label{eq: monotonicity formula 1}
	\begin{split}
	&\rho^{-n}_2\mathcal{H}^{n}(\Sigma_{a,t}\cap B^{n+1}_{\rho_2}(\xi))-	\rho^{-n}_1\mathcal{H}^{n}(\Sigma_{a,t}\cap B^{n+1}_{\rho_1}(\xi))\\
    =&\rho^{-n}_2 \int_{\Sigma_{a,t}\cap B^{n+1}_{\rho_2}(\xi)}
    \left(|(\nabla_M r)^{\perp}|^2-v\right)
	-\rho^{-n}_1 \int_{\Sigma_{a,t}\cap B^{n+1}_{\rho_1}(\xi)}
    \left(|(\nabla_M r)^{\perp}|^2-v\right)\\
	+&n\int^{\rho_2}_{\rho_1}s^{-1-n} \left(\int_{\Sigma_{a,t}\cap B^{n+1}_{s}(\xi)}|(\nabla_M r)^{\perp}|^2-v\right)ds
	+\int^{\rho_2}_{\rho_1}s^{-1-n}(\int_{\Sigma_{a,t}\cap B^{n+1}_{s}(\xi)}u)ds.\\
    \end{split}
\end{equation}
By co-area formula, for any $C^0$ function $\varphi$
we have
\[
\frac{d}{ds}\int_{\Sigma_{a,t} \cap B^{n+1}_s(\xi)}\varphi|\nabla_{\Sigma_{a,t}} r|
=\int_{\Sigma_{a,t} \cap \partial B^{n+1}_s(\xi)}\varphi,~~ \text{for a.e. $s\in [\rho_1, \rho_2]$}.
\]
Take $\varphi=(|(\nabla_M r)^{\perp}|^2-v)|\nabla_{\Sigma_{a,t}} r|^{-1}$. Then  multiply $s^{-n}$ on both sides of the equation above and take integral on $[\rho_1,\rho_2]$, we have
\begin{equation}\label{eq: first term}
\begin{split}
&\int^{\rho_2}_{\rho_1}s^{-n}
    \int_{\Sigma_{a,t}\cap \partial B^{n+1}_{s}(\xi)}(|(\nabla_M r)^{\perp}|^2-v)|\nabla_{\Sigma_{a,t}} r|^{-1}ds\\
    =&\rho^{-n}_2 \int_{\Sigma_{a,t}\cap B^{n+1}_{\rho_2}(\xi)}
    (|(\nabla_M r)^{\perp}|^2-v)
	-\rho^{-n}_1 \int_{\Sigma_{a,t}\cap B^{n+1}_{\rho_1}(\xi)}
   ( |(\nabla_M r)^{\perp}|^2-v)\\
	&+n\int^{\rho_2}_{\rho_1}s^{-1-n}\int_{\Sigma_{a,t}\cap B^{n+1}_{s}(\xi)}(|(\nabla_M r)^{\perp}|^2-v)ds.\\
    \end{split}
\end{equation}
On the other hand, set 
$$
y(s):=\int_{\Sigma_{a,t}\cap  B^{n+1}_{s}(\xi))}|v|d\mu,
$$
then  by co-area formula,  for a.e. $s\in[\rho_1,\rho_2]$, we have
$$
y'(s)=\int_{\Sigma_{a,t}\cap \partial B^{n+1}_{s}(\xi)}|v||\nabla_{\Sigma_{a,t}} r|^{-1} d\mu.
$$
It follows that
\begin{equation}\label{eq: estimate term1}
\begin{split}
    \int^{\rho_2}_{\rho_1}s^{-n}
    \int_{\Sigma_{a,t}\cap \partial B^{n+1}_{s}(\xi)}|v||\nabla_{\Sigma_{a,t}} r|^{-1}ds
=& \int^{\rho_2}_{\rho_1}s^{-n} y'(s)ds \\
\leq& \rho^{-n}_2y(\rho_2)+n\int^{\rho_2}_{\rho_1}s^{-n-1}
    \int_{\Sigma_{a,t}\cap B^{n+1}_{s}(\xi)}|v|ds.\\
\end{split}
\end{equation}
Substituting \eqref{eq: first term} and \eqref{eq: estimate term1}
into \eqref{eq: monotonicity formula 1} gives \eqref{eq: monotonicity formula}.
\end{proof}

In Lemma \ref{monotonicity outside compact set} and Lemma \ref{lem: almost monotonicity across compact set} below, we establish the necessary monotonicity properties to obtain density estimates in $\Sigma_{a,t}$. Then we prove the main result Proposition \ref{prop: density estimate1} in this subsection .

\begin{lemma}\label{monotonicity outside compact set}

(Monotonicity outside $B^{n+1}_{\rho_0}(O)$)\\
Given $\delta\in(0,1)$, there exists $\rho_0 = \rho_0(M,g,\delta)\geq R_0$, such that for any $\xi\in \Sigma_{a,t}$ with $|\xi|\geq \rho_0$ and any $\rho_1, \rho_2$ with $\delta\leq \rho_1<\rho_2$, if  $(B^{n+1}_{\rho_2}(\xi)\setminus B^{n+1}_{\rho_1}(\xi))\subset M\backslash B^{n+1}_{\rho_0}(O)$ and $B^{n+1}_{\rho_2}(\xi)\cap\partial\Sigma_{a,t}=\emptyset$,
then
\begin{equation}\label{eq: density estimate}
  \theta_{\xi}(\rho_2,a)- \theta_{\xi}(\rho_1,a)\ge -\delta.  
\end{equation}
\end{lemma}
\begin{proof} 
We introduce a function $\Phi$ given by
\begin{equation}
 \Phi(x)=\left\{
\begin{aligned}
&|x|^{-\frac{3}{2}},\ \ \quad \text{for $x\in M\setminus B^{n+1}_{R_0}(O)$ },\\
& 1,\quad\ \quad \ \ \  \ \text{for $x\in B^{n+1}_{R_0}(O)$ }.
\end{aligned}
\right.\nonumber
\end{equation}
Note that $\tau>\frac{n}{2}\geq\frac{3}{2}$. According to the definition of $u$ and $v$ in the proof of Lemma \ref{lmm: monotonicity ineq2}, we see that there exists some uniform $C$ such that 
\begin{equation}\label{eq: bound u,v}
|u|+n|v|\leq C|\xi|\Phi.    
\end{equation}
We consider the following two cases.

\textbf{Case 1:} $B^{n+1}_{\rho_1}(\xi)\subset B^{n+1}_{\rho_2}(\xi)\subset M\backslash B^{n+1}_{\rho_0}(O)$.\\
Using $v=O(x^{-\tau})$ outside $B^{n+1}_{\rho_0}(O)$ and the minimality of $\Sigma_{a,t}$,
we have
\begin{equation}\label{eq: ft for v}
  \rho^{-n}_2
  \int_{\Sigma_{a,t}\cap  B^{n+1}_{\rho_2}(\xi)}|v|d\mu\leq
 \rho_0^{-\tau}\rho^{-n}_2 \mathcal{H}^n(\Sigma_{a,t}\cap B^{n+1}_{\rho_2}(\xi))
  \leq
  C\rho_0^{-\tau}\leq\frac{\delta}{2}
\end{equation}
if we choose $\rho_0\geq C^2\delta^{-1}$. 

\textbf{Claim:}  There exists some uniform $C$ depending only on $(M^{n+1}, g)$ such that 
\begin{equation}\label{eq: decay estimate1}
  \int_{\Sigma_{a,t}\cap B^{n+1}_{s}(\xi)}\Phi=\int_{\Sigma_{a,t}\cap B^{n+1}_{s}(\xi)}|x|^{-\frac{3}{2}}
 \leq C|\xi|^{-\frac{3}{2}}s^{n}.   
\end{equation}
We only need to consider the case that $|\xi|\leq 2s$,
otherwise we immediately have 
\begin{equation}\label{eq: decay estimate2}
 \int_{\Sigma_{a,t}\cap B^{n+1}_{s}(\xi)}|x|^{-\frac{3}{2}}
 \leq \mathcal{H}^n(\Sigma_{a,t}\cap B^{n+1}_{s}(\xi))(\frac{|\xi|}{2})^{-\frac{3}{2}}
 \leq C|\xi|^{-\frac{3}{2}}s^n.  
\end{equation}
Since $\Sigma_{a,t}$ is area-minimizing and $(M^{n+1},g)$ is asymptotically flat, we have
\begin{equation}
 \int_{\Sigma_{a,t}\cap B^{n+1}_{s}(\xi)}|x|^{-\frac{3}{2}}
 \leq C\int_{\Sigma_{a,t}\cap B^{n+1}_{s}(\xi)}|x|^{-\frac{3}{2}}d\bar{\mu},  
\end{equation}
where $d\bar{\mu}$ is the area element of $\Sigma_{a,t}$ in $(\mathbf{R}^{n+1}, \delta_{ij})$. 
By Lemma 65 in \cite{EK23} we have
\begin{equation}
\begin{split}
  \int_{\Sigma_{a,t}\cap B^{n+1}_{s}(\xi)}|x|^{-\frac{3}{2}}d\bar{\mu}
  \leq&\int_{\Sigma_{a,t}\cap (B^{n+1}_{|\xi|+s}(O)\setminus B^{n+1}_{|\xi|-s}(O))}|x|^{-\frac{3}{2}}d\bar{\mu}\\
    \leq&  C(|\xi|+s)^{-\frac{3}{2}}+C \Large((|\xi|+s)^{n-\frac{3}{2}}+(|\xi|-s)^{n-\frac{3}{2}}\Large)\\
    \leq&C|\xi|^{-\frac{3}{2}}s^n.
\end{split} 
\end{equation}
 Note $|\xi|\geq \rho_0+\rho_2$ and $\rho_1\geq\delta$.
By choosing 
$\rho_0\geq C\delta^{-3}$, we have
\begin{equation}\label{eq: last term}
\begin{split}
   \int^{\rho_2}_{\rho_1}s^{-1-n}(\int_{\Sigma_{a,t}\cap B^{n+1}_{s}(\xi)}\Phi)ds
    \leq& C|\xi|^{-\frac{3}{2}}\int^{\rho_2}_{\rho_1}s^{-1}ds\\
    \leq& C|\xi|^{-\frac{3}{2}}(\ln\rho_2-\ln\rho_1)\\
    \leq&C|\xi|^{-\frac{3}{2}}(\ln|\xi|-\ln\delta)\\
    <&C^{-1}|\xi|^{-1}\frac{\delta}{2}. 
\end{split}
\end{equation}
Plugging \eqref{eq: bound u,v}, \eqref{eq: ft for v} and \eqref{eq: last term}
into \eqref{eq: monotonicity formula} gives
 the desired estimate.\\
\textbf{Case 2:} $B^{n+1}_{\rho_0}(O)\subset B^{n+1}_{\rho_1}(\xi)\subset B^{n+1}_{\rho_2}(\xi)$.\\
In this case, $\rho_2>\rho_1\geq \rho_0+|\xi|$.
we have
\begin{equation}\label{eq: estimate phi 2}
\begin{split}
      \int_{\Sigma_{a,t}\cap B^{n+1}_{s}(\xi)}\Phi
      =& 
      \int_{\Sigma_{a,t}\cap B^{n+1}_{s} (\xi)\cap B^{n+1}_{R_0} (O)}\Phi
      +\int_{\Sigma_{a,t}\cap (B^{n+1}_{s} (\xi)\setminus B^{n+1}_{R_0} (O))}\Phi\\
      \leq& \mathcal{H}^{n}(\Sigma_{a,t}\cap B^{n+1}_{R_0}(O))+\int_{\Sigma_{a,t}\cap (B^{n+1}_{s} (\xi)\setminus B^{n+1}_{R_0} (O))}|x|^{-\frac{3}{2}}\\
      \leq&C(R_0^{n}+( |\xi|+s )^{n-\frac{3}{2}}),
\end{split}
\end{equation}
where we have used Lemma 65 in \cite{EK23} in the last inequality.
It follows that 
\begin{equation}\label{eq: ft for v 2}
\begin{split}
    \rho^{-n}_2
  \int_{\Sigma_{a,t}\cap  B^{n+1}_{\rho_2}(\xi))}|v|d\mu
  \leq& C\rho_2^{-n}\int_{\Sigma_{a,t}\cap B^{n+1}_{s}(\xi)}\Phi\\
  \leq&C\rho_2^{-n}(R_0^{n}+( |\xi|+\rho_2 )^{n-\frac{3}{2}})\\
  \leq&\frac{\delta}{2}  
\end{split} 
\end{equation}
by our choice of $\rho_0$.
A straightforward calculation shows
\begin{equation}\label{eq: estimate term2}
   \int_{\rho_1}^{\rho_2} s^{-1-n}(|\xi|+s)^{n-\frac{3}{2}}ds
\le C\int_{\rho_1}^{\rho_2}s^{-\frac{5}{2}}ds\le C\rho_1^{-\frac{3}{2}}<C^{-1}|\xi|^{-1}\frac{\delta}{2}. 
\end{equation}
Substituting  \eqref{eq: bound u,v} and  \eqref{eq: estimate phi 2}--\eqref{eq: estimate term2} into \eqref{eq: monotonicity formula} yields \eqref{eq: density estimate}.
\end{proof}

Next, establish an  almost monotonicity across the compact set $B_{\rho_0}^{n+1}(O)$.
\begin{lemma}\label{lem: almost monotonicity across compact set}(Almost monotonicity across the compact set)\\
    Let $\delta,\rho_0$ be as in Lemma \ref{monotonicity outside compact set}. Then there exists $L>2$ depending only on $\delta$ such that for any 
    $\xi\in \Sigma_{a,t}\setminus  B^{n+1}_{L\rho_0}(O) $ and  any $\rho_1, \rho_2$ with $\delta\leq \rho_1<\rho_2$, if  $B^{n+1}_{\rho_1}(\xi)\cap B^{n+1}_{\rho_0}(O)=\emptyset$, $B^{n+1}_{\rho_0}(O)\subset B^{n+1}_{\rho_2}(\xi)$ and $B^{n+1}_{\rho_2}(\xi)\cap\partial\Sigma_{a,t}=\emptyset$, then
    \begin{align*}
        \theta_{\xi}(\rho_1,a)\le (1+\delta)\theta_{\xi}(\rho_2,a)+3\delta
    \end{align*}
\end{lemma}
\begin{proof}
By our assumption, $|\xi|-\rho_0\geq\rho_1$ and $|\xi|+\rho_0\leq \rho_2$.
Divide $[\rho_1,\rho_2]$ into $[\rho_1,|\xi|-\rho_0]\cup[|\xi|-\rho_0,|\xi|+\rho_0]\cup[|\xi|+\rho_0,\rho_2]$. Then
    \begin{align*}
         B^{n+1}_{\rho_1}(\xi)\subset B^{n+1}_{|\xi|-\rho_0}(\xi)\subset M\backslash B^{n+1}_{\rho_0}(O)\ \ \text{and}\ \ 
         B^{n+1}_{\rho_0}(O)\subset B^{n+1}_{|\xi|+\rho_0}(\xi)\subset B^{n+1}_{\rho_2}(\xi).
    \end{align*}
    So by applying Lemma \ref{monotonicity outside compact set} we have
    \begin{equation}\label{eq: 3}
    \theta_{\xi}(\rho_1,a)\le \theta_{\xi}(|\xi|-\rho_0,a)+\delta\ \ \ \text{and}
\ \ \ \theta_{\xi}(|\xi|+\rho_0,a)\le \theta_{\xi}(\rho_2,a)+\delta.        
    \end{equation}
    Notice that
    \begin{equation}\label{eq: 4}
        \begin{split}
          \theta_{\xi}(|\xi|-\rho_0,a)\leq \theta_{\xi}(|\xi|+\rho_0,a) \frac{(|\xi|+\rho_0)^{n}}{(|\xi|-\rho_0)^{n}}
         \leq \theta_{\xi}(|\xi|+\rho_0,a) (\frac{L+1}{L-1})^{n}
        \le& (1+\delta)\theta_{\xi}(|\xi|+\rho_0,a) 
        \end{split}
    \end{equation}
    when $L$ is sufficiently large. Combining \eqref{eq: 3} with \eqref{eq: 4} we get the desired estimate.
\end{proof}

By combining Lemma \ref{monotonicity outside compact set} and Lemma \ref{lem: almost monotonicity across compact set}, we are ready to prove the main proposition concerning density estimate in this subsection.

\begin{proposition}\label{prop: density estimate1}
  Let $\delta\in(0,\frac{1}{100}),L$ and $\rho_0$  be as above.
  Then there exists $a_0>L\rho_0>0$, such that for any $a>a_0$ and $\xi\in \Sigma_{a,t}\backslash B^{n+1}_{L\rho_0}(O)$ with $d(\xi,\partial\Sigma_{a,t})\geq \frac{a}{2}$, there holds
    \begin{align}\label{eq: density estimate2}
        \theta_{\xi}(\rho,a)\le 1+20\delta
    \end{align}
    for any $ \rho>0$ with $B^{n+1}_{\rho}(\xi)\cap\partial\Sigma_{a,t}=\emptyset$.
\end{proposition}
\begin{proof}
We briefly sketch the proof. We divide the argument into two cases: 
$\rho>\delta$ and $\rho\le\delta$. The case $\rho>\delta$ is further 
divided into two subcases. The first case is $|\xi|<\epsilon a$, where $\epsilon$ is a fixed small 
constant determined by $\delta$. In this case we apply 
Lemma \ref{monotonicity outside compact set} and 
Lemma \ref{lem: almost monotonicity across compact set} to show that 
the density $\theta_\xi(\rho,a)$ is comparable to 
$\theta_\xi((1-\epsilon)a,a)$. The latter can in turn be compared with 
$\theta_O(a,a)=1+o(1)$. The second case is $|\xi|>\epsilon a$. In this case $\xi$ can be detected 
via a blow-down argument, which allows us to conclude the proof of the 
case $\rho>\delta$. For the case $\rho\le\delta$, we apply the local monotonicity inequality 
in Lemma \ref{lm: monotonicity ineq2} to reduce the problem to the case 
$\rho>\delta$.

Let's begin with the $\rho>\delta$ case. We adopt the contradiction argument to  show  there exists $a_0>L\rho_0$, such that
 \begin{align}\label{eq: density estimate3}
        \theta_{\xi}(\rho,a)\le 1+8\delta \ \ \ \text{for}\ \ 
        \rho\geq\delta.
    \end{align}
whenever $a>a_0$. Suppose not, then there exists a sequence of $a_i\rightarrow\infty$, and $\xi_i$ with 
    $|\xi_i|>L\rho_0$ and $\rho_i$ with $\delta\leq\rho_i$ such that
\begin{equation}\label{eq: area growth}
   \theta_{\xi_i}(\rho_i,a_i)> 1+8\delta.  
\end{equation}
We consider the following two cases:

    \textbf{Case 1:} There is a subsequence(still denoted by $i$)
    such that $L\rho_0<|\xi_i|<\epsilon a_i$ where  $\varepsilon$ is given by $1+\delta=(1-\varepsilon)^{-n}$. In this case, choose
    $\tilde{\rho}_i=(1-\epsilon)a_i$, without loss of generality, we may assume
    $\tilde{\rho}_i\geq \rho_i$.
    Then
    by Lemma \ref{monotonicity outside compact set} and  Lemma \ref{lem: almost monotonicity across compact set},  
    \begin{equation}\label{eq: little center}
       \theta_{\xi_i}(\tilde{\rho}_i,a_i)
       \geq (1+\delta)^{-1}\theta_{\xi_i}(\rho_i,a_i)-3\delta
       \geq 1+2\delta.
    \end{equation}
    Since $\Sigma_{a_i,t}$ is the area-minimizing hypersurface in asymtotically flat manifold
    with $\partial\Sigma_{a_i,t}=\partial D_{a_i,t}$, we have
    \begin{equation}
    \begin{split}
  \lim_{i\rightarrow\infty}\theta_{\xi_i}(\tilde{\rho}_i,a_i)=&
   \lim_{i\rightarrow\infty}\frac{\mathcal{H}^{n}(B^{n+1}_{\tilde{\rho}_i}(\xi_i)\cap\Sigma_{a_i,t})}{\omega_{n}\tilde{\rho}_i^{n}}\\
   \leq& \lim_{i\rightarrow\infty}\frac{\mathcal{H}^{n}(B^{n+1}_{(1-\varepsilon)a_i}(\xi_i)\cap\Sigma_{a_i,t})}{\omega_{n} a_i^{n-1}}\frac{1}{(1-\varepsilon)^{n}}\\
\leq& 1+\delta,
\end{split}
    \end{equation}
which is incompatible with \eqref{eq: little center}.  
    
    \textbf{Case 2:} $\epsilon a_i<|\xi_i|$. In this case, by Lemma \ref{monotonicity outside compact set}, we may choose some $\tilde{\rho}_i\geq \varepsilon a_i$ such that
    \begin{equation}\label{eq: rescale radius}
        \theta_{\xi_i}(\tilde{\rho}_i,a_i)\geq1+2\delta.
    \end{equation}
    We adopt the blow down argument.   
Now $\{(M^{n+1}\setminus B^{n+1}_1(O), a_i^{-2}g)\}$ converges to $(\mathbf{R}^{n+1}\setminus\{O\}, \delta_{ij})$
in the $C^{\infty}_{loc}$ sense as $i\rightarrow \infty$. 
Meanwhile, $\{(a_i^{-1}\Sigma_{a_i,t},a_i^{-1}\xi_i)\}$  converges to an area-minimizing hypersurface $(\bar{\Sigma}, \bar{\xi})$
with $\partial \bar{\Sigma}=\partial B^{n+1}_1(O)\cap S_0$, in both current and varifold sense; see \cite[Theorem 34.5]{Simon83}.  Thus, $\bar{\Sigma}$ must be a part of  the hyperplane. The condition $|\xi_i|>\epsilon a_i$ implies $|\bar{\xi}|\ge\epsilon>0$, which means $\xi\in\bar{\Sigma}$ is a point other than the origin $O$. Note that the density of $\bar{\Sigma}$ at $\xi$ is $1$, $i.e.$
\[
\frac{\mathcal{H}^{n}(B^{n+1}_{s}(\bar{\xi})\cap\bar{\Sigma})}{\omega_{n}s^{n}}= 1
\ \ \ \text{for}\ \ 0<s<1-|\bar\xi|,
\]
which contradicts with \eqref{eq: rescale radius} due to the local varifold convergence $a_i^{-1}\Sigma_{a_i,t}\longrightarrow\bar{\Sigma}$ in $\mathbf{R}^{n+1}\backslash\{O\}$.

$\quad$

Next, we use Lemma \ref{lm: monotonicity ineq2} to deal with the case $\rho\leq\delta$. By taking some larger $\rho_0$ if necessary we have for $0<\rho\leq\delta$
\begin{equation}\label{eq: area comparison}
     \mathcal{H}^{n}(B^{n+1}_\rho(\xi)\cap\Sigma_{a,t})\leq
(1+\delta)\mathcal{H}^{n}(\mathcal{B}^{n+1}_\rho(\xi)\cap\Sigma_{a,t})
\leq  (1+\delta)^2\mathcal{H}^{n}(B^{n+1}_\rho(\xi)\cap\Sigma_{a,t}),
\end{equation}
whenever $|\xi|>L\rho_0$. Here $\mathcal{B}^{n+1}_\rho(\xi)$ denotes the geodesic ball in $(M^{n+1},g)$.
Then, using Lemma \ref{lm: monotonicity ineq2} we have
 \begin{equation}\label{eq: area comparison2}
    \begin{split}
    \theta_{\xi}(\rho,a)=&\frac{\mathcal{H}^{n}(B^{n+1}_{\rho}(\xi)\cap\Sigma_{a,t})}{\omega_{n}\rho^{n}}\\
    \leq & (1+\delta)\frac{\mathcal{H}^{n}(\mathcal{B}^{n+1}_{\rho}(\xi)\cap\Sigma_{a,t})}{\omega_{n}\rho^{n}}\\
    \leq& (1+\delta)^2\frac{\mathcal{H}^{n}(\mathcal{B}^{n+1}_{\delta}(\xi)\cap\Sigma_{a,t})}{\omega_{n}\delta^{n}} \\
    \leq&(1+\delta)^3\frac{\mathcal{H}^{n}(B^{n+1}_{\delta}(\xi)\cap\Sigma_{a,t}) }{\omega_{n}\delta^{n}}.
    \end{split} 
 \end{equation}
 Together \eqref{eq: density estimate3} with \eqref{eq: area comparison2}  gives the desired estimate.
Therefore, we complete the proof.
\end{proof}

\subsection{Curvature estimate for area-minimizing  hypersurfaces  in asymptotically flat manifolds}

$\quad$

\subsubsection{Smoothness of $\Sigma_{a,t}$ outside a compact set}

$\quad$

In the following, we establish the smoothness of $\Sigma_{a,t}$ outside a compact set. To achieve this, we prove a density estimate for $\Sigma_{a,t}$ at a point $\xi$ after $M^{n+1}$ is embedded into a higher dimensional Euclidean space $\mathbf{R}^{n+k}$, based on Proposition \ref{prop: density estimate1} established above. The estimate here is purely local and may depend on $\xi$, which yields the regularity of $\Sigma_{a,t}$ in a neighborhood of $\xi$ with an application of Allard's regularity theorem.

Recall that $B^{n+1}_\rho(x)$ denotes the coordinate ball in the AF end 
$E\subset M^{n+1}$, and $\mathcal{B}^{n+1}_\rho(x)$ denotes the geodesic ball 
in $(M^{n+1},g)$. Since our argument also involves $\mathbf{R}^{n+k}$, we 
follow the same notation convention and denote by $B^{n+k}_\rho(x)$ the 
coordinate ball in $\mathbf{R}^{n+k}$. The reader can distinguish 
$B^{n+k}_\rho(x)$ from $B^{n+1}_\rho(x)$ by the superscript: the former is 
defined in $\mathbf{R}^{n+k}$, while the latter is defined in $M^{n+1}$.

Given $\delta\in(0,\frac{1}{100})$ and let $L,\rho_0,a_0$ be as in Proposition \ref{prop: density estimate1}. Fix a point $\xi\in \Sigma_{a,t}\backslash B^{n+1}_{L\rho_0}(O)$ with $d(\xi,\partial\Sigma_{a,t})\geq \frac{a}{2}$, there is an isometric embedding
 $$
 i: (B^{n+1}_{1}(\xi),g)\hookrightarrow \mathbf{R}^{n+k}, ~\text{for some $k = k(n)>1$}. 
 $$
 from Nash embedding theorem \cite[Theorem 3]{Nash1956}.
 By taking $\theta_0$ sufficiently small (possibly dependent on $\xi$), we have
\begin{align*}
         i^{-1}(B_{\rho}^{n+k}(\xi))\subset\mathcal{B}_{(1+\delta)\rho}^{n+1}(\xi)\subset B_{(1+2\delta)\rho}^{n+1}(\xi)
\end{align*}
whenever $\rho<\theta_0$. Here the first inclusion can be seen since the image of $M$ under the inclusion map $i$ in $\mathbf{R}^{n+k}$ is sufficiently close to the tangent space $T_\xi M\subset \mathbf{R}^{n+k}$ near $\xi$, while the second one is obtained from the asymptotical flatness of $E$ and that $|\xi|>\rho_0$. Therefore, we obtain
\begin{equation}\label{eq: comsarison for ES}
    \mathcal{H}^{n}(\Sigma_{a,t}\cap B^{n+k}_{\rho} (\xi))\leq \mathcal{H}^{n}(\Sigma_{a,t}\cap B^{n+1}_{\rho(1+2\delta)} (\xi))
\end{equation}
for any $\rho< \theta_0$. 
\begin{lemma}\label{lem: density estimate in high dimensional Euclidean space}
There exists some $c_n$ depending only on $n$ such that  it holds
 \begin{equation}\label{eq: area ratio}
    \frac{\mathcal{H}^{n}(B^{n+k}_{\rho}(\xi)\cap\Sigma_{a,t})}{\omega_{n}\rho^{n}}
    \leq1+c_n\delta\ \ \  \ \text{for}  \ \rho\ \ \text{ small enough.}
\end{equation}
\end{lemma}
\begin{proof}
By \eqref{eq: comsarison for ES} and Proposition \ref{prop: density estimate1}, we have
\begin{equation}
   \begin{split}
\frac{\mathcal{H}^{n}(\Sigma_{a,t}\cap B^{n+k}_{\rho}(\xi))}{\omega_{n}\rho^{n}}&
\leq
\frac{\mathcal{H}^{n}(\Sigma_{a,t}\cap B^{n+1}_{(1+2\delta)\rho}(\xi))}{\omega_{n}\rho^{n}}
\le (1+2\delta)^n\frac{\mathcal{H}^{n}(\Sigma_{a,t}\cap B^{n+1}_{(1+2\delta)\rho}(\xi))}{\omega_{n}(1+2\delta)^n\rho^{n}}\\
&\leq(1+2\delta)^n(1+20\delta)\le 1+c_n\delta.
 \end{split}  
\end{equation}
\end{proof}

\begin{lemma}\label{lem: generalized mean curvature}
Let $\Omega\subset (M^{n+1},g)$ be a bounded domain, with an isometric embedding $i:\Omega\longrightarrow \mathbf{R}^{n+k}$. Let $\Sigma\subset M^{n+1}$ be a stationary $n$-varifold, $\mathbf{H}$ the generalized mean curvature of $\Sigma \cap \Omega$ in $\mathbf{R}^{n+k}$, and $\mathbf{A}$ the second fundamental form of the isometric embedding $i:\Omega\longrightarrow\mathbf{R}^{n+k}$. Then for $\mathcal{H}^n$-a.e. $x \in \Sigma$, it holds
$$
\sup_{\Omega\cap \Sigma}|\mathbf{H}|\leq n\sup_{\Omega}|\mathbf{A}|.
$$
\end{lemma}
\begin{proof}
	Let  $X:\Sigma \to \mathbf{R}^{n+k}$ be any $C^1$- vector field on $\Sigma$ with compact support set in $\Omega$,  then
	$$
	X=X^{\perp}+X^{\top},
	$$
where $X^{\perp}$ and $X^{\top}$	 denote the normal and tangential components of $X$  to $(M^{n+1},g)$ respectively. Since $\Sigma$ is area-minimizing in $(M^{n+1},g)$, we have
\[
\int div_{\Sigma} X^{\top} d\mathcal{H}^n=0.
\]
Let $\{\tau_i\}_{1\leq i\leq n}$ be any orthonormal basis for the approximate tangent space $T_x \Sigma$. Then 
$$
 div_{\Sigma}X^{\perp}=\sum_{i=1}^{n}\langle\tau_i, \nabla_{\tau_i} X^{\perp}\rangle
 =\sum_{i=1}^{n}\langle\mathbf{A}(\tau_i, \tau_i), X^{\perp}\rangle
 $$
and
$$
\int div_\Sigma X d\mathcal{H}^n=\int \langle\sum_{i=1}^{n-1}\mathbf{A}(\tau_i, \tau_i), X \rangle d\mathcal{H}^n.
$$
Hence,
$$
\mathbf{H}=\sum_{i=1}^{n}\mathbf{A}(\tau_i, \tau_i),
$$
which implies 
$$
\sup_{B^{n+k}_\rho(x)\cap \Sigma}|\mathbf{H}|\leq n\sup_{B^{n+k}_\rho(x)\cap M^{n+1}}|\mathbf{A}|.
$$
\end{proof}
For given $\delta$,  by Lemma \ref{lem: generalized mean curvature}, we can take 
 $\rho$ small enough such that 
 \[
 \left(\rho^{p-n}\int_{B^{n+k}_\rho(\xi)\cap \Sigma_{a,t}}|\mathbf{H}|^p\right)^{\frac{1}{p}}\leq \delta
 \]
 for $p> n$. In conjunction with the density estimate Lemma \ref{lem: density estimate in high dimensional Euclidean space} and applying the Allard's regularity theorem to $\Sigma_{a,t}\cap B^{n+k}_{\rho}(\xi)\subset M\subset\mathbf{R}^{n+k}$, we conclude: 
\begin{corollary}\label{cor: smooth outside compact set}
    Let the notation be as in the Proposition \ref{prop: density estimate1}. Then for all $a>a_0$, $\Sigma_{a,t}$ is smooth at any point $\xi\in \Sigma_{a,t}\backslash B^{n+1}_{L\rho_0}(O)$ with $d(\xi,\partial\Sigma_{a,t})\geq \frac{a}{2}$.
\end{corollary}

\subsubsection{Curvature estimate of $\Sigma_{a,t}$ outside a compact set}

$\quad$

In the rest of this subsection, we use the standard point-picking argument and density estimate Proposition \ref{prop: density estimate1} to derive the following curvature estimate.

\begin{lemma}\label{lem: uniform estimate for A}
For given $t\in\mathbf{R}$, there exist  uniform constants $R_1$ and $R_2$, such that for any $a>R_1$ and
    any point $x\in\Sigma_{a,t}\setminus B^{n+1}_{R_2}(O)$ with $d(x,\partial \Sigma_{a,t})\geq\frac{3a}{4}$,
    it holds $|A|(x)\leq \frac{C}{|x|}$ for some constant $C$ depending only on $(M^{n+1}, g)$. Recall $|A|$ denotes the second fundamental form of $\reg\Sigma^n$ in $(M^{n+1},g)$.
\end{lemma}
\begin{proof}
    Suppose the lemma is not true, then there exist a sequence of $a_k\rightarrow\infty$ and $x_k\in\Sigma_{a_k,t}$
    with $|x_k|\rightarrow\infty$ and $d(x_k, \partial\Sigma_{a_k,t})\geq \frac{3a_k}{4}$, such that 
    \[
\max_{x\in\Sigma_{a_k,t}\cap B^{n+1}_{\frac{|x_k|}{2}}(x_k)}(\frac{|x_k|}{2}-|x-x_k|)|A|(x):=c_k
\to\infty.
    \]
    Assume the maximum is attained at $y_k$ and set $r_k=\frac{|x_k|}{2}-(|y_k-x_k|)$.
    Rescale the asymptotically flat metric $g_{ij}$ in $B^{n+1}_{r_k}(y_k)$ to the metric $\tilde{g}_{ij}^k$
    using the transformation $\Phi$ given by $x=y_k+\frac{r_k}{c_k}\tilde{x}$, i.e.,
    \[\tilde{g}_{ij}^k=\Phi^*g_{ij}.
    \]
    Then $\tilde{g}_{ij}^k$ is defined in $B^{n+1}_{c_k}(\tilde{O})$ and converge to the Euclidean metric $\delta_{ij}$ on $\mathbf{R}^{n+1}$ locally uniformly in $C^2$ sense by asymptotic flatness, since $|x_k|\to\infty$. 
    The second fundamental form of the rescaled hypersurface $\Phi^*(\Sigma_{a_k,t}\cap B^{n+1}_{r_k}(y_k))$ satisfies $|\tilde{A}|\leq 1$
    with $|\tilde{A}|(\tilde{O})=1$.
Thus, by passing to a subsequence, we have $\Phi^*(\Sigma_{a_k,t}\cap B^{n+1}_{r_k}(y_k))$ converge to a
smooth
area-minimizing hypersurface $\tilde{\Sigma}$. 

On the other hand, by Proposition \ref{prop: density estimate1}, there also exists a sequence of $\delta_k\rightarrow 0$ such that
\begin{equation}\label{eq: asy density esimate}
        \frac{\mathcal{H}^{n}(B^{n+1}_{\rho}(y_k)\cap\Sigma_{a_k,t})}{\omega_{n}\rho^{n}}\le 1+20\delta_k,\mbox{ for }\rho>0 \ \text{with}
        \ \ B^{n+1}_{\rho}(y_k)\cap\partial\Sigma_{a_k,t}=\emptyset.
    \end{equation}
It follows that
\[
\frac{\mathcal{H}^{n}(B^{n+1}_{R}(\tilde{O})\cap\tilde{\Sigma})}{\omega_{n}R^{n}}\equiv 1,
\mbox{ for any}\ \ R> 0.
\]
By the monotonicity formula for minimal hypersurface in Euclidean space,  $\tilde{\Sigma}$ must be the hyperplane, which contradicts with $|\tilde{A}|_{\tilde{\Sigma}}(\tilde{O})=1$.   
\end{proof}

\subsection{Foliation of area-minimizing hypersurfaces in higher dimension}

$\quad$

Let $\Sigma_{r, t}$ be given by Proposition \ref{prop: solving plateau}. Now we use catenoidal-type
barriers devised by Schoen and Yau\cite{SY81} to give a nice control of $\Sigma_{r,t}$.

For $\Lambda>1$ and $1<\beta<\tau-1$, consider the rotational symmetric function $f(r)$ in $\mathbf{R}^{n}\setminus B^{n}_1(O)$ defined by
\begin{equation}\label{eq: symmetric function}
f(r)=\Lambda \int_{r}^{+\infty}(s^{2\beta+2}-\Lambda^2)^{-\frac{1}{2}}ds,\ \ 
\text{where}\ \  r>\Lambda.
\end{equation}
Denote the graph of $f+L$ by $\ggraph_{f+L}$ for $L\in \mathbf{R}$. From the calculation of \cite[equation (2.6)]{HSY24}, $\ggraph_{f+L}$ has positive mean curvature respect to the upward normal whenever $\Lambda>\Lambda_0$, where $\Lambda_0$ is a constant depending only on the geometry of $(E,g|_E)$. Similarly, $\ggraph_{-f+L}$ has positive mean curvature respect to the downward normal.

Arguing  as in \cite[Lemma 2.5]{HSY24}, we obtain for any given $t$, $\Sigma_{r, t}\backslash C_\Lambda$ lies between $\ggraph_{\pm f+t}$ for any $r\gg 1$.
Since $(M^{n+1},g)$ has bounded geometry and $\Sigma_{r, t}\cap \partial C_{\Lambda}$ lies between two hyperplanes $S_{t\pm C\Lambda}$, by applying the monotonicity formula for minimal hypersurfaces we conclude that $\Sigma_{r, t}\cap C_{\Lambda}$ is bounded between two hyperplanes $S_{t\pm C_1\Lambda}$ for some uniform $C_1$.

With this at hand we're able to construct the area-minimizing hypersurfaces in $(M,g)$. Choose a sequence of $r_i$ with $r_i\rightarrow\infty$, then for some fixed $t$, $\Sigma_{r_i,t}$ intersects with a fixed compact set that depends only on $t$. For any bounded domain $W\subset\subset M$, the minimizing property of $\Sigma_{r,t}$ implies $\mathcal{H}^n(\Sigma_{r,t}\cap W)\le\mathcal{H}^n(\partial^*W)$. Hence, from Theorem \cite[Theorem 27.3]{Simon83}, $\Sigma_{r_i,t}$ converges to an area-minimizing hypersurface $\Sigma_t$ as $r_i\to +\infty$ in current sense by taking a subsequence if necessary.  Moreover, since we have nice second fundamental form estimate Lemma \ref{lem: uniform estimate for A}, the convergence is also locally smooth outside $B_{R_2}(O)$, where $R_2$ is as in Lemma \ref{lem: uniform estimate for A}. We summarize these as follows:
 \begin{corollary}\label{cor: bound ms}
 Let $\Sigma_t$ be the area-minimizing hypersurface as above. Then $\Sigma_t$ lies between $S_{t\pm C_1\Lambda}$ for some uniform $C_1$ and can be bounded by $\ggraph_{\pm f+t}$ outside the cylinder $C_{\Lambda}$. Moreover, $\Sigma_t$ is smooth outside a fixed compact set $\mathbf{K} = B_{R_2}(O)$.
 \end{corollary}

\begin{remark}\label{rmk: no drift to infinity for arbitrary ends}
If $(M^{n+1},g)$ has ends other than $E$, the Plateau solution in Proposition 
\ref{prop: solving plateau} may not exist for all $t$. Indeed, when $|t|$ is small, 
$\Sigma_{a,t}$ may drift off to infinity along the arbitrary end due to the lack of 
geometric control. We modify the argument above to show that this cannot happen 
when $|t|$ is sufficiently large. Without loss of generality, we assume $t>0$ in 
the following.

In Proposition \ref{prop: solving plateau}, we solve the Plateau problem 
$\Sigma_{r,t}$ in the AF end 
$E\cong \mathbf{R}^{n+1}\setminus B_1^{n+1}(O)$ with inner obstacle 
$\partial E\cong \partial B_1^{n+1}(O)$ satisfying 
$\partial\Sigma_{r,t} = S_{r,t}$. Then the intersection of $\Sigma_{r,t}$ with 
the interior of $E$ is minimal. If $\Sigma_{r,t}\cap\partial E = \emptyset$, 
then $\Sigma_{r,t}$ is a minimal hypersurface and the previous argument applies 
without change. 

If $\Sigma_{r,t}\cap\partial E \ne \emptyset$, since 
$\Sigma_{r,t}\cap \partial C_\Lambda$ lies between $S_{t\pm C\Lambda}$, the 
monotonicity formula yields
\begin{align}\label{eq:6}
\mathcal{H}^n(\Sigma_{r,t}\cap \{z\le t-C\Lambda\})
\ge \eta (t-C\Lambda)+C_1 ,
\end{align}
where $\eta$ and $C_1$ depend only on the geometry of $(E,g|_E)$; see for example 
\cite[Lemma B.1]{HSY24}. On the other hand, a comparison argument gives
\begin{equation}\label{eq:7}
\begin{split}
\mathcal{H}^n(\Sigma_{r,t}\cap C_\Lambda)
&\le \mathcal{H}^n (\partial C_\Lambda\cap \{t-C\Lambda\le z\le t+C\Lambda\}) 
   + \mathcal{H}^n(S_{t+C\Lambda}\cap C_\Lambda)\\
&\le 2C\Lambda\cdot n\omega_n \Lambda^{n-1}
   + \omega_n\Lambda^n+O(\Lambda^{n-1}) \\
&= ((1+2nC)\omega_n+o(1))\Lambda^n .
\end{split}
\end{equation}
Combining \eqref{eq:6} and \eqref{eq:7}, we conclude that 
$\Sigma_{r,t}\cap \partial E = \emptyset$ provided $t>C_2$, where $C_2$ is a 
constant depending only on $(E,g|_E)$.
\end{remark}
 
Now we are ready to show that 
\begin{proposition}\label{prop: uniqueness of tangent cone}
    Let $\Sigma_t$ be an area-minimizing hypersurface as above. Then the tangent cone of $\Sigma_t$ at the infinity is regular and unique.
\end{proposition}
	\begin{proof}
	  Since $\Sigma_t$  has local volume bound and $|A|(x)\leq\frac{C}{|x|}$ for $|x|>R_2$ in Lemma \ref{lem: uniform estimate for A}, for any sequence
      $\{r_k\}$ with $r_k\to\infty$, $r^{-1}_k\Sigma_t$ converges locally smoothly to some stable minimal cone $\Gamma$ in $\mathbf{R}^{n+1}$ by passing to a subsequence if necessary. As $\Sigma_t$ lies between two hyperplanes, the limit $\Gamma$ must be supported in the hyperpalne $\{x_{n+1}=0\}$. Thus, the constancy theorem \cite[Theorem 26.27]{Simon83} implies that $\Gamma$ equals $m[|\{x_{n+1}=0\}|]$ for some integer $m$ in current sense. The area estimate for $\Sigma_t$ implies $\Gamma$ has multiplicity one.
	\end{proof}
Applying Theorem 5.7 in \cite{Simon84}(see also the discussion given at pp. 269-270),  $\Sigma_t$ can be represented by some graph 
$\{(y,u_t(y):y\in \mathbf{R}^{n}\}$ outside some compact set with
$|u_t(y)-t|\rightarrow 0$ as $|y|\rightarrow \infty$.
We end this subsection by giving the proof of Theorem \ref{thm: foliation}.
	
	\begin{proof}[Proof Theorem \ref{thm: foliation}]
First we deal with the case that $M$ has only one end. By the argument before, for given $t\in \mathbf{R}$, there exists an area-minimizing hypersurface $\Sigma_t$ asymptotic to the hyperplane $S_t$. The asymptotic estimate for the graph function $u_t$ follows from Proposition 9 in \cite{EK23}. We can use the same argument in \cite{HSY24} to show $\Sigma_t$ is the unique area-minimizing hypersurface asymptotic to the hyperplane $S_t$ for $|t|\gg1$ and these $\Sigma_t$ form $C^1$ foliation (see Proposition 2.13 and
Proposition 2.14 in \cite{HSY24} for details).

 In the scenario where the manifold $M$ possesses arbitrary ends and  has bounded geometry, for each $t\in \mathbf{R}$ and  large $r$, we can use the usual comparison arguments to find a minimal boundary $\Sigma_{r,t}\subset C_r$ with $\partial\Sigma_{r,t}=S_{r,t}$ and least boundary volume. Moreover, for each AF end $E$, $\Sigma_{r,t}\setminus E$ is contained in a fixed compact set of $M^{n+1}$.
 Then, by employing the catenoidal-type barrier argument delineated prior to Corollary \ref{cor: bound ms}, we can still conclude that $\Sigma_{r_i,t}$ lies between $S_{t\pm C\Lambda}$ for $|t|>2C\Lambda$. Then we can follow a similar argument to get the desired result.

 Finally, we consider the case where $M$ has arbitrary ends and does not necessarily have bounded geometry. In this case, $\Sigma_t$ may exist only for $|t|$ sufficiently large; see Remark \ref{rmk: no drift to infinity for arbitrary ends}. Using the same argument as above, we obtain a smooth foliation $\{\Sigma_t\}_{|t|\ge T_0}$.
\end{proof}

\section{Proof of Theorem \ref{thm: 8dim Schoen conj}}
In this section, we  give the proof of Theorem \ref{thm: 8dim Schoen conj}. Since the ambient manifold $(M^{n+1},g)$ may have arbitrary end, we need to introduce a sequence of \textit{free boundary problems with inner obstacle in cylinders} to effectively compensate for the lack of compactness in minimal surfaces caused by arbitrary ends.

Recall $C_{R}$ is the coordinate cylinder  with radius $R$. Let
\begin{align*}
    \partial E\subset V_1\subset V_2\subset\dots
\end{align*}
be any  compact exhaustion of  $M\setminus E$ and let $j:\mathbb{R}\longrightarrow \mathbb{Z}_+$ be any non-decreasing index function with 
\begin{align*}
    \lim_{R\to\infty} j(R) = \infty.
\end{align*}
\begin{figure}
    \centering
    \includegraphics[width=12cm]{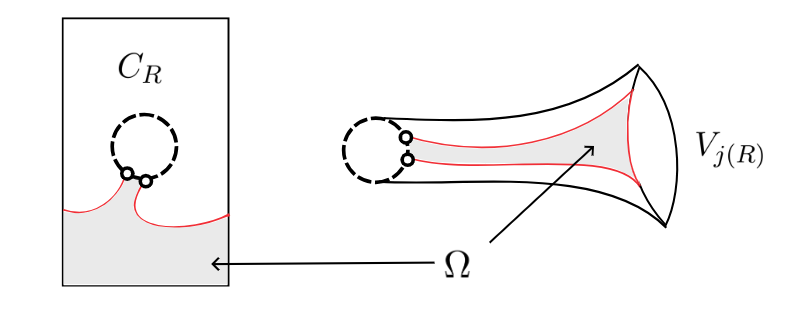}
    \caption{This figure demonstrates the free boundary problem with inner obstacle defined by \eqref{eq: 79}. The left rectangle denotes part of the cylinder $C_R$; the gray area denotes $\Omega$; the red line denotes $\Sigma$,  and $\partial V_{R_i}$ is the inner obstacle lying in arbitrary ends.}
    \label{f5}
\end{figure}
 We define(see Figure $2$)
\begin{equation}\label{eq: 79}
\begin{split}
\mathcal{F}_{R}&:=\{\Sigma=\partial \Omega \backslash \partial C_{R}: \Omega \subset C_{R}\cup V_{j(R)} \text{  is a Caccioppoli set that satisfies}\\
	&\text{ $C_{R} \cap \{z\leq a\}\subset \Omega $ and $\Omega \cap \{z\geq b\}= \emptyset  $ for some $-\infty< a\leq b<\infty$} \}.
	\end{split}
\end{equation}

Suppose there is a sequence of compact exhaustion  $ \partial E\subset V_1\subset V_2\subset\dots$ of  $M\setminus E$  and a sequence $\{R_i\}$ with $R_i\rightarrow\infty$ as $i\rightarrow\infty$ such that the corresponding free boundary minimizing hypersurfaces with inner obstacle $ \Sigma_i:=\Sigma_{R_i}=\partial \Omega_i\setminus \partial C_{R_i}$ exist and satisfy $\Sigma_i\cap K\neq\emptyset$ for some compact $K$. By \cite[Corollary 12.27]{Mag12}, $\Sigma_i$ converges subsequentially to a limit area-minimizing boundary $\Sigma$.
\begin{lemma}\label{lem: asymtotic to hyperplane}
  $\Sigma\cap E$  is asymptotic to a hyperplane  which is parallel to $S_0$ at infinity in Euclidean sense.  
\end{lemma}
\begin{proof}
    For $A\subset M$, we define the height $Z(A)$ of $A$ by
\begin{equation}
    Z(A)=\left\{
    \begin{aligned}
     &\sup_{p\in A\cap E}z(p)-\inf_{p\in A\cap E}z(p),\ \ A\cap E\neq\emptyset,\\
    &0,\ \ A\cap E=\emptyset.   
    \end{aligned}
    \right.
\end{equation}
By the same argument as in \cite[Lemma 3.2]{HSY24}, $Z(\Sigma_i)=o(R_i)$. 

Now we can argue as in Section 2 to prove the free boundary solution $\Sigma_i$ has similar properties as the Plateau solutions considered in Section 2. In fact, the argument in Section 2 made use of two key properties of the Plateau solutions $\Sigma_{a,t}$, which also hold true for the free boundary solution $\Sigma_i$:
\begin{itemize}
    \item 
        $\mathcal{H}^n(\Sigma_{a_i,t}\cap E)\le (1+o(1))\omega_na_i^n$ as $a_i\to\infty$.
    \item The blow down limit $a_i^{-1}(\Sigma_{a_i,t}\cap E)$ is a punctured unit ball $(B_1^{n+1}(O)\cap S_0)\backslash \{O\}$.
\end{itemize}
Therefore, from the argument in Proposition \ref{prop: density estimate1}, we can show that for given $\delta>0$, there exists some $c_{\delta}$ depending on $(M^{n+1}, g)$ and $\delta$, such that for 
any $\xi\in\Sigma_i\cap E$ with $|\xi|\geq c_{\delta}$ and $d(\xi,\partial\Sigma_i)\geq \frac{R_i}{2}$,
it holds
\[
\frac{\mathcal{H}^{n}(B^{n+1}_{\rho}(\xi)\cap\Sigma_i)}{\omega_{n}\rho^{n}}\leq 1+20\delta
\]
 for any $\rho$ with $B^{n+1}_{\rho}(\xi)\cap\partial\Sigma_i\cap(\partial C_{R_i})=\emptyset$ when $i$ is large enough. Then we use the same argument as in  Lemma \ref{lem: uniform estimate for A} to show the existence of a constant $C$ such that
\begin{align}\label{(1)}
    |A|(x)\leq \frac{C}{|x|}\text{ for }x\in\Sigma\cap E\text{ with }|x|\gg1.
\end{align}

Thus, we can apply the argument in the proof of \cite[Proposition 8]{EK23} to conclude that 
there exists some rotation $S\in SO(n+1)$ and constant $a\in \mathbf{R}$ such that  outside some compact set, $S(\Sigma\cap(E\setminus B^{n+1}_{r_o}(O)))=\{(y, u(y)):y\in\mathbf{R}^{n}\} $ with $u(y)$ satisfying $|u(y)-a|=O(|y|^{1-\tau+\varepsilon})$ for any $\varepsilon>0$.
The dimensional assumption $3 \le n+1 \le 7$ in \cite{EK23} was used solely to obtain curvature estimates for $\Sigma\cap E$ via the Schoen–Simon curvature estimate \cite{Schoen1981}. In our setting, although we allow arbitrary dimensions, the curvature estimate has already been established in \eqref{(1)}. Therefore, their argument applies verbatim. 
Note that $Z(\Sigma_i)=o(R_i)$. Then  the rotation $S$ must be  the identity. We complete the proof.
\end{proof}

Note $\Sigma$ may also have arbitrary ends and even contain isolated singular  set $\mathcal{S}$ when $n=7$. 
We recall the following notion of \textit{strong stability}. This notion was first introduced by Schoen in his proof of the PMT for AF manifolds with dimensions no greater than $7$ via dimension reduction by non-compact minimal hypersurfaces \cite{Schoen1989}, and later studied in \cite{Carlotto16,EK23,HSY24,HSY26}. More specifically, we have
\begin{definition}\label{defn: strong test function}
		Let $(N^{n},g)$ be  a manifold (not necessarily complete) containing an AF end $E$. We say that a locally Lipschitz function $\phi$ is an \textit{asymptotically constant test function} on $N$, if it is supported on a neighborhood $\mathcal{U}$ of $E$ such that $\mathcal{U}\Delta E$ is compact, satisfying $\lim\limits_{|x|\to\infty, x\in E}\phi = 1$ and $\phi-1\in W^{1,2}_{-q}(\mathcal{U})$, for some $q> \frac{n-2}{2}$. i.e.
		$$
		\int_{\mathcal{U}}|\phi-1|^2 r^{2q-n}d\mu_g+\int_{\mathcal{U}}|\nabla\phi|^2 r^{2(q+1)-n}d\mu_g<\infty,
		$$
		where $r$ denotes a positive function on $\mathcal{U}$ and $r=|x|$ on $E$.
        \end{definition}

\begin{definition}\label{defn: strongly stable hypersurface}
    Let $\Sigma$ be an area-minimizing hypersurfcae in $M^{n+1}$ which is  asymptotic to a  coordinate hyperplane in the AF end $E$ of $M^{n+1}$ with $\Sigma\cap E$ being an AF end for $\Sigma$. We say $\Sigma$ is  strongly stable, if for any function $\phi\in \Lip_{loc}(\Sigma\backslash\mathcal{S})$ which is either compactly supported or asymptotically constant in the sense of Definition \ref{defn: strong test function}, there holds
	\begin{equation}\label{eq: strongly stability}
		\int_{\Sigma\backslash\mathcal{S}}|\nabla\phi|^2-(Ric(\nu,\nu)+|A|^2)\phi^2\geq 0. 
	\end{equation}
    Here $\nu$ is the unit normal vector of $\Sigma^{n}$ in $(M^{n+1},g)$.
\end{definition}

 We now proceed to present the proof of Theorem \ref{thm: 8dim Schoen conj} by adopting the strategy in \cite{HSY24} to establish the existence of the strongly stable hypersurface $\Sigma_p$ passing through $p$ for any $p\in E$ with $|z(p)|\gg1$. The new point lies in the application of the following Theorem \ref{prop: rigidity for minimal surface}  to help us show that all those $\Sigma_p$ are totally geodesic and isometric to $\mathbf{R}^n$.

\begin{theorem}[Theorem 1.1 in \cite{HSY26}]\label{prop: rigidity for minimal surface}
		Let $(M^{n+1},g)$ be an AF manifold with arbitrary ends, an AF end  $E$ of asymptotic order  $\tau>n-2$  and $R_g\ge 0$. Let $\Sigma^{n}\subset M^{n+1}$ be an area-minimizing boundary in $M^{n+1}$ with isolated singular set $\mathcal{S}$. Assume $\Sigma$ is  strongly stable, and one of the following holds:
		
		(a) $n+1\le 8$
		
		(b) $\Sigma\backslash\mathcal{S}$ is spin,  and the tangent cone at each point in $\mathcal{S}$ has isolated singularity.
		
		$\quad$
		
		(1) If $\Sigma\backslash E$ is compact, then we have $\mathcal{S} = \emptyset$, and $\Sigma$ is isometric to $\mathbf{R}^{n}$, with $R_g = |A|^2 = Ric(\nu,\nu) = 0$ along $\Sigma$.
		
		(2) If there exists $U_1,U_2\subset M$ with $E\subset U_1\subset U_2$, such that $U_i\backslash E$ is compact ($i = 1,2$) and  $R_g> 0$ on $U_2\backslash U_1$, then $\Sigma\subset U_1$, and the same conclusion in (1) holds. 
	\end{theorem}

        \begin{proof}[Proof of Theorem \ref{thm: 8dim Schoen conj}]
        We take the contradiction argument. Suppose not, then by the argument in the paragraph behind \eqref{eq: 79} there exists an area-minimizing boundary $\Sigma$ such that $\Sigma\cap E$  is asymptotic to a hyperplane  which is parallel to $S_0$ at infinity in Euclidean sense. We divide our proof into 3 steps.
        
        \textbf{Step 1: $\Sigma$ is strongly stable in the sense of Definition \ref{defn: strongly stable hypersurface}.} 
        
The strong stability of $\Sigma$ follows from the calculation in \cite{Schoen1989}, see also \cite[Lemma 3.6]{HSY24}.

$\quad$
       
        \textbf{Step 2: Construct  a family of  area minimizing boundaries in a  new class of free boundary under perturbed metrics.}

        $\quad$
        
        We first apply Theorem \ref{prop: rigidity for minimal surface} to show in both cases stated in Theorem \ref{thm: 8dim Schoen conj}, $\Sigma\backslash E$ is a compact set and $\Sigma$ is isometric to $\mathbf{R}^{n}$.

        \begin{enumerate}
            \item If $(M,g)$ only has AF ends, then it is geometrically bounded. This implies $\Sigma\backslash E$ is a compact set. Thus, we can apply Theorem \ref{prop: rigidity for minimal surface} (1) to conclude $\Sigma$ is isometric to $\mathbf{R}^{n}$.

            \item If $(M,g)$ has some arbitrary end other than $E$, we then apply Theorem \ref{prop: rigidity for minimal surface} (2) to conclude  $\Sigma\backslash E$ is a compact set, under the scalar curvature assumption $R_g>0$ in $U_2\backslash \bar{U}_1$. Therefore, Theorem \ref{prop: rigidity for minimal surface} (2) yields  $\Sigma$ is isometric to $\mathbf{R}^{n}$.
        \end{enumerate}
        
         Since $\Sigma$ and $\Sigma_{i}$ have multiplicity one and $\Sigma_i$ converges to $\Sigma$ in varifold sense, using the Allard's regularity theorem, we know that for any bounded domain $L$, $\Sigma_i\cap L$ must be smooth for sufficiently large $i$. Note that $\Sigma$ separates the AF end $E$ into two parts: the upper part $E_+$ and the lower part $E_-$. Let $p$ be a point in $E_+$ with $|z(p)|\geq T_0$, where $T_0$ is given by Theorem \ref{thm: foliation}. Then as in \cite[Lemma 3.3]{HSY24}, for any $r_0>0$, there exists $0<r<r_0$, an open set $W\subset M$ with compact closure that satisfies $W\cap \Sigma\ne\emptyset$, and a family of Riemannian metrics $\lbrace g(s)\rbrace_{s\in[0,1]}$, such that the following holds

            (1) $g(s)\to g$ smoothly as $s\to 0$,

            (2) $g(s)=g$ in $M\backslash W$,

            (3) $g(s)<g$ in $W$,

            (4) $R(g(s))>0$ in $\lbrace x\in W: \dist(x,p)>r\rbrace$,

            (5) For $i$ sufficiently large, $\Sigma_i$ is weakly mean convex and strictly mean convex at one interior point under $g(s)$ with respect to the normal vector pointing into $E_-$.

        Let $C_{R_i}^+$ denote the closure of the region in $C_{R_i}$ that lies  above $\Sigma_i\cap E$ and
         \begin{equation}\label{eq:Gr}
        	\begin{split}
        \mathcal{G}_{R_i} = &\lbrace \Sigma = \partial\Omega\setminus  \partial C_{R_i}^+  : ~\Omega \mbox{ is a Caccioppoli set in } C_{R_i}^+\cup V_i,\\
      &\Sigma_i\subset\Omega\quad\text{and}\quad\Omega\cap\{z\ge b\} = \emptyset \mbox{ for some }b\rbrace
      \end{split}
    \end{equation}
    Similar to \cite[Lemma 3.4]{HSY24}, we are able to prove that there exists a free boundary minimal hypersurface $\Sigma_i(s)$ which minimizes the volume in $\mathcal{G}_{R_i}$ under the metric $g(s)$.  Here, a slight difference is that $\Sigma_i$ may have singularity. However, thanks to \cite[Theorem B.1]{Wang24}, $\Sigma_i$ remains as an effective barrier. Moreover, by the argument in \cite[Lemma 3.5]{HSY24}, for a fixed $s$, there exists a compact set $W_0\subset W$, such that for all sufficiently large $i$, there holds $\Sigma_i(s)\cap W_0\ne\emptyset$.
    Here the free boundary minimal surfaces $\Sigma_i(s)$
		may also have singularities. Now by the standard result in geometric measure theory, $\Sigma_i(s)$ will converge to a limit area-minimizing hypersurface $\Sigma(s)$.

$\quad$
        
    \textbf{Step 3:
    Show that $m_{ADM}(M,g,E)=0$,  which contradicts the assumption that $m_{ADM}(M,g,E)\neq0$.}
    
    By \cite[Lemma 3.6]{HSY24}, $\Sigma(s)\cap E$ is asymptotic to a hyperplane  which is parallel to $S_0$ at infinity in Euclidean sense and $\Sigma(s)$ is strongly stable. By the desingularizing Theorem \ref{prop: rigidity for minimal surface}, we conclude that $\Sigma(s)\cap B^{n+1}_r(p)\ne\emptyset$. By letting $s\rightarrow 0$ and $r\rightarrow 0$, we obtain a strongly stable area-minimizing hypersurface $\Sigma_p$ in $(M^{n+1}, g)$ that passes through $p$. The uniqueness result in Theorem \ref{thm: foliation} shows that $\Sigma_p$ coincides with $\Sigma_t$. Again due to Theorem \ref{prop: rigidity for minimal surface}, $\Sigma_p$ is isometric to $\mathbf{R}^{n}$ with $R_g = |A|^2 = Ric(\nu,\nu) = 0$ along $\Sigma$. Combined with \cite[Proposition A.3]{HSY24}, this implies that the region  in the AF end $E$ that lies above $\Sigma_{T_0}$ is isometric to $\mathbf{R}^{n+1}_+$. Similarly, the region  in the AF end $E$ that lies below $\Sigma_{-T_0}$ is also isometric to $\mathbf{R}^{n+1}_-$. By the definition of the ADM mass, we see that 
    $m_{ADM}(M,g,E)=0$ and get the desired contradiction.
        \end{proof}

{\bf Data availability statement.} Data sharing not applicable to this article as no datasets were generated or analysed during the current study.

{\bf Conflict of interest.} On behalf of all authors, the corresponding author states that there is no conflict of interest.

\bibliographystyle{alpha}

\bibliography{Positive}

@book{CM11,
	author = {Colding, Tobias Holck and Minicozzi, II, William P.},
	doi = {10.1090/gsm/121},
	isbn = {978-0-8218-5323-8},
	mrclass = {53A10 (35J93 49Q05)},
	mrnumber = {2780140},
	mrreviewer = {Andrew\ Bucki},
	pages = {xii+313},
	publisher = {American Mathematical Society, Providence, RI},
	series = {Graduate Studies in Mathematics},
	title = {A course in minimal surfaces},
	url = {https://doi.org/10.1090/gsm/121},
	volume = {121},
	year = {2011}}

@article {SY79PNAS,
    AUTHOR = {Schoen, Richard M. and Yau, Shing Tung},
     TITLE = {Complete manifolds with nonnegative scalar curvature and the
              positive action conjecture in general relativity},
   JOURNAL = {Proc. Nat. Acad. Sci. U.S.A.},
  FJOURNAL = {Proceedings of the National Academy of Sciences of the United
              States of America},
    VOLUME = {76},
      YEAR = {1979},
    NUMBER = {3},
     PAGES = {1024--1025},
      ISSN = {0027-8424},
   MRCLASS = {58D30 (53C50 83C99)},
  MRNUMBER = {524327},
MRREVIEWER = {Mauro\ Francaviglia},
       DOI = {10.1073/pnas.76.3.1024},
       URL = {https://doi.org/10.1073/pnas.76.3.1024},
}

@book {Giu1984,
    AUTHOR = {Giusti, Enrico},
     TITLE = {Minimal surfaces and functions of bounded variation},
    SERIES = {Monographs in Mathematics},
    VOLUME = {80},
 PUBLISHER = {Birkh\"auser Verlag, Basel},
      YEAR = {1984},
     PAGES = {xii+240},
      ISBN = {0-8176-3153-4},
   MRCLASS = {58E12 (49F10 53A10)},
  MRNUMBER = {775682},
MRREVIEWER = {Helmut\ Kaul},
       DOI = {10.1007/978-1-4684-9486-0},
       URL = {https://doi.org/10.1007/978-1-4684-9486-0},
}

@book {Simon83,
    AUTHOR = {Simon, Leon},
     TITLE = {Lectures on geometric measure theory},
    SERIES = {Proceedings of the Centre for Mathematical Analysis,
              Australian National University},
    VOLUME = {3},
 PUBLISHER = {Australian National University, Centre for Mathematical
              Analysis, Canberra},
      YEAR = {1983},
     PAGES = {vii+272},
      ISBN = {0-86784-429-9},
   MRCLASS = {49-01 (28A75 49F20)},
  MRNUMBER = {756417},
MRREVIEWER = {J.\ S.\ Joel},
}

@article {Carlotto16,
    AUTHOR = {Carlotto, Alessandro},
     TITLE = {Rigidity of stable minimal hypersurfaces in asymptotically
              flat spaces},
   JOURNAL = {Calc. Var. Partial Differential Equations},
  FJOURNAL = {Calculus of Variations and Partial Differential Equations},
    VOLUME = {55},
      YEAR = {2016},
    NUMBER = {3},
     PAGES = {Art. 54, 20},
}

@article{Liu13,
	author = {Liu, Gang},
	journal = {Inventiones mathematicae},
	number = {2},
	pages = {367--375},
	title = {3-Manifolds with nonnegative Ricci curvature},
	volume = {193},
	year = {2013}}

@misc{EK23,
	archiveprefix = {arXiv},
	author = {Michael Eichmair and Thomas Koerber},
	eprint = {2303.12200},
	primaryclass = {math.DG},
	title = {Schoen's conjecture for limits of isoperimetric surfaces},
	year = {2023}}

@article {HSY24,
    AUTHOR = {He, Shihang and Shi, Yuguang and Yu, Haobin},
     TITLE = {Foliation of area minimizing hypersurfaces in asymptotically
              flat manifolds and {S}choen's conjecture},
   JOURNAL = {Calc. Var. Partial Differential Equations},
  FJOURNAL = {Calculus of Variations and Partial Differential Equations},
    VOLUME = {64},
      YEAR = {2025},
    NUMBER = {2},
     PAGES = {Paper No. 48},
      ISSN = {0944-2669,1432-0835},
   MRCLASS = {53C21 (53C24)},
  MRNUMBER = {4847302},
       DOI = {10.1007/s00526-024-02911-5},
       URL = {https://doi.org/10.1007/s00526-024-02911-5},
}

@article{CCE16,
	author = {Carlotto, Alessandro and Chodosh, Otis and Eichmair, Michael},
	fjournal = {Inventiones Mathematicae},
	issn = {0020-9910,1432-1297},
	journal = {Invent. Math.},
	number = {3},
	pages = {975--1016},
	title = {Effective versions of the positive mass theorem},
	volume = {206},
	year = {2016}}

@article {Schoen1981,
    AUTHOR = {Schoen, Richard and Simon, Leon},
     TITLE = {Regularity of stable minimal hypersurfaces},
   JOURNAL = {Comm. Pure Appl. Math.},
  FJOURNAL = {Communications on Pure and Applied Mathematics},
    VOLUME = {34},
      YEAR = {1981},
    NUMBER = {6},
     PAGES = {741--797},
      ISSN = {0010-3640,1097-0312},
   MRCLASS = {49F22 (53C42 58E15)},
  MRNUMBER = {634285},
MRREVIEWER = {F.\ J.\ Almgren, Jr.},
       DOI = {10.1002/cpa.3160340603},
       URL = {https://doi.org/10.1002/cpa.3160340603},
}

@article {Schoen1989,
AUTHOR = {Schoen, Richard M.},
     TITLE = {Variational theory for the total scalar curvature functional
              for {R}iemannian metrics and related topics},
 BOOKTITLE = {Topics in calculus of variations ({M}ontecatini {T}erme,
              1987)},
    SERIES = {Lecture Notes in Math.},
    VOLUME = {1365},
     PAGES = {120--154},
 PUBLISHER = {Springer, Berlin},
      YEAR = {1989},
      ISBN = {3-540-50727-2},
   MRCLASS = {58E11 (49F99 53C20 58D17 58G30)},
  MRNUMBER = {994021},
MRREVIEWER = {Hubert\ Gollek},
       DOI = {10.1007/BFb0089180},
       URL = {https://doi.org/10.1007/BFb0089180},
}

@article {SY81,
    AUTHOR = {Schoen, Richard and Yau, Shing Tung},
     TITLE = {Proof of the positive mass theorem. {II}},
   JOURNAL = {Comm. Math. Phys.},
  FJOURNAL = {Communications in Mathematical Physics},
    VOLUME = {79},
      YEAR = {1981},
    NUMBER = {2},
     PAGES = {231--260},
      ISSN = {0010-3616,1432-0916},
   MRCLASS = {83C99 (35J60 53A10 53C50 58G40 83C05)},
  MRNUMBER = {612249},
MRREVIEWER = {J.\ L.\ Kazdan},
       URL = {http://projecteuclid.org/euclid.cmp/1103908964},
}

@incollection {Simon84,
    AUTHOR = {Simon, Leon},
     TITLE = {Isolated singularities of extrema of geometric variational
              problems},
 BOOKTITLE = {Harmonic mappings and minimal immersions ({M}ontecatini,
              1984)},
    SERIES = {Lecture Notes in Math.},
    VOLUME = {1161},
     PAGES = {206--277},
 PUBLISHER = {Springer, Berlin},
      YEAR = {1985},
      ISBN = {3-540-16040-X},
   MRCLASS = {58E15 (58E20)},
  MRNUMBER = {821971},
MRREVIEWER = {Harold\ Parks},
       DOI = {10.1007/BFb0075139},
       URL = {https://doi.org/10.1007/BFb0075139},
}

@article{LUY21,
    author = {Lesourd, Martin and Unger, Ryan and Yau, Shing-Tung},
    title = {The positive mass theorem with arbitrary ends},
    journal = { J. Differential Geom. },
    year = {2024}
}

@article{Wang24,
    author = {Zhihan Wang},
    title = {Mean Convex Smoothing of Mean Convex Cones},
    journal = {Geom. Funct. Anal.},
    year = {2024},
    volume = {34},
    pages = {263--301}
}

@article {EK24,
    AUTHOR = {Eichmair, Michael and Koerber, Thomas},
     TITLE = {Foliations of asymptotically flat manifolds by stable constant
              mean curvature spheres},
   JOURNAL = {J. Differential Geom.},
  FJOURNAL = {Journal of Differential Geometry},
    VOLUME = {128},
      YEAR = {2024},
    NUMBER = {3},
     PAGES = {1037--1083},
      ISSN = {0022-040X,1945-743X},
   MRCLASS = {53C20 (53C42)},
  MRNUMBER = {4810218},
MRREVIEWER = {Juan\ A.\ Aledo},
       DOI = {10.4310/jdg/1729092454},
       URL = {https://doi.org/10.4310/jdg/1729092454},
}

@article{EM13,
    author = { Michael Eichmair and Jan Metzger},
	journal = {Invent. Math.},
	year = {2013},
    pages = {591-630},
    volume = {194},
	title = {Unique isoperimetric foliations of asymptotically flat manifolds in all dimensions},
}

@book{Mag12,
    author = {Francesco Maggi},
    title = {Sets of finite perimeter and geometric variational problems, An introduction to geometric measure theory},
    publisher = {Cambridge Studies in Advanced Mathematics},
    year = {2012}
}

@misc{HSY26,
      title={Singularity removal rigidity theorems for minimal hypersurfaces in manifolds with nonnegative scalar curvature}, 
      author={Shihang He and Yuguang Shi and Haobin Yu},
      year={2026},
      eprint={2602.23705},
      archivePrefix={arXiv},
      primaryClass={math.DG},
      url={https://arxiv.org/abs/2602.23705}, 
}

@article{Nash1956,
author = {Nash, John},
title = {The Imbedding Problem for Riemannian Manifolds},
journal = {Annals of Mathematics},
volume = {63},
number = {1},
year = {1956},
pages = {20--63}
}

\end{document}